\documentclass[article,11pt]{amsart}
\usepackage[latin1]{inputenc}
\usepackage[T1]{fontenc}
\usepackage{amsmath,pifont}
\usepackage{amsthm}
\usepackage{amsfonts}
\usepackage{amssymb}
\usepackage{amsthm}
\usepackage{mathrsfs}
\usepackage{enumitem,epic,eepic,pstricks,tikz}
\usepackage{graphicx}
\usepackage{footnote}
\usepackage{hyperref}

\renewcommand{\geq}{\geqslant}
\renewcommand{\leq}{\leqslant}
\newtheorem{thm}{Theorem}

\newtheorem{conj}{Conjecture}
\newtheorem{cor}[thm]{Corollary}
\newtheorem{prop}[thm]{Proposition}
\newtheorem{lem}[thm]{Lemma}

\usepackage{bbding}

\usepackage{color}
\definecolor{darkgreen}{rgb}{0,0.4,0}
\definecolor{MyDarkBlue}{rgb}{0,0.08,0.50}
\definecolor{BrickRed}{rgb}{0.65,0.08,0}

\hypersetup{
colorlinks=true,       
    linkcolor=blue,          
    citecolor=red,        
    filecolor=BrickRed,      
    urlcolor=darkgreen        
}

\title[{Walks in the quarter plane, harmonic functions and conformal mappings}]{Random walks in the quarter plane,\\discrete harmonic functions and\\conformal mappings}

\author[Kilian Raschel (with an appendix by Sandro Franceschi)]{Kilian Raschel\\ \\(With an appendix by Sandro Franceschi)} 
\address{F\'ed\'eration Denis Poisson \& Laboratoire de Math\'ematiques et Physique Th\'eorique (Universit\'e de Tours, France) \& Laboratoire de Probabilit\'es et Mod\`eles Al\'eatoires (Universit\'e Pierre et Marie Curie, France)}
\email{sfrances@clipper.ens.fr}
\address{CNRS \& F\'ed\'eration Denis Poisson \& Laboratoire de Math\'ematiques et Physique Th\'eorique (Universit\'e de Tours, France)}
\email{Kilian.Raschel@lmpt.univ-tours.fr}


\date{\today}

\begin{document}

\begin{abstract}
We propose a new approach for finding discrete harmonic functions in the quarter plane with Dirichlet conditions. It is based on solving functional equations that are satisfied by the generating functions of the values taken by the harmonic functions. As a first application of our results, we obtain a simple expression for the harmonic function that governs the asymptotic tail distribution of the first exit time for random walks from the quarter plane. As another corollary, we prove, in the zero drift case, the uniqueness of the discrete harmonic function. 
\end{abstract}

\maketitle

\section{Introduction}
\setcounter{equation}{0}

\subsection*{Context}
Random processes conditioned on staying in cones of ${\bf Z}^d$ arouse a great interest in the mathematical community, as they appear in several distinct domains: quantum random walks \cite{Bi1,Bi3}, random matrices \cite{Dy62}, non-colliding random walks \cite{DW1,DW,EK}, etc. A usual way to realize this conditioning consists in using Doob $h$-transforms, thanks to functions which are harmonic for the process, positive within the cone and equal to zero elsewhere---or equivalently harmonic and positive for the underlying killed process. It is therefore natural to be interested in finding some (all, if possible) positive harmonic functions for processes in cones of ${\bf Z}^d$ killed at the first exit time from the cone. 

In the literature, results on harmonic functions for killed processes in cones are broken up, and most of them concern very particular cones, as half spaces ${\bf Z}_+\times {\bf Z}^{d-1}$ and orthants ${\bf Z}_+^d$. Regarding random walks with {\it non-zero drift}, general results are obtained in \cite{IRL,KuRa} for the above domains when $d=2$: the Martin boundary is found, and happens to be composed of infinitely many harmonic functions. For random walks with {\it zero drift}, the results are rare, and typically require a strong underlying structure: the random walks are cartesian products in \cite{PW}; they are associated with Lie algebras in \cite{Bi1,Bi3}; certain reflexion groups are supposed to be finite in \cite{RaSp4}. This is problematic, as these harmonic functions are useful to construct other important processes (see above). Last but not least, knowing the harmonic functions for zero drift random walks in ${\bf Z}_+^{d-1}$ is necessary for building the harmonic functions of walks with drift in ${\bf Z}_+^d$, see \cite{IR2}.

In this article, we prove that for the whole class of walks with small steps and zero drift killed at the boundary of the quadrant ${\bf Z}_+^2$, there exists exactly one non-zero discrete harmonic function (up to multiplicative constants). Further, the unique harmonic function is expressed in terms of sine and arcsine functions. Our approach\footnote{We wish to thank D.\ Denisov and V.\ Wachtel for suggesting us this new approach, during a workshop at Eindhoven University (The Netherlands) in February 2012.} is new---to the best of our knowledge---and is based on functional equations that satisfy the generating functions of the values taken by the harmonic functions.

Before presenting it (in point \ref{possibility-functional-equation} below), we notice that in the literature, there are several ways to obtain harmonic functions:

\renewcommand{\thesubsubsection}{\arabic{subsubsection}}
\subsubsection{}\label{possibilit-discrete-laplacian}The most elementary approach just uses the definition of a harmonic function: if $P$ denotes the generator of the Markov chain under consideration, we must have $P f =f$ inside of the cone, and $f=0$ outside. Unfortunately, except for a few particular cases (see, e.g., \cite{Woess}), solving directly this discrete Dirichlet problem is out of range.

\subsubsection{} \label{possibility-Martin-boundary} A second one, based on Martin boundary theory, is used in \cite{Bi1,IR2,IR1,IRL,KuRa,RaSp4}. Briefly said, given a (transient) Markov process, this theory provides explicit expressions (typically, integral representations) for all its (super)harmonic functions. To find the Martin boundary, it suffices \cite{Dynkin} to compute the asymptotic behavior of quotients of Green functions. More details are given in Section \ref{subsec:MB}. The main and profound difficulty of this approach lies in finding the (asymptotics of quotients of) Green functions.

\subsubsection{}\label{possibility-tail-distribution} A third one, which is studied in \cite{DW1,DW,EK,KS1}, is based on the fact that in many examples, the asymptotic tail distribution of the exit time $\tau$ from the cone of the random walk started at $x$ is governed by a certain positive harmonic function $V(x)$: 
     \begin{equation}
     \label{eq:asymptotic-DW}
          {\bf P}_x[\tau>n] = \varkappa V(x)F(n)(1+o(1)),\qquad n\to\infty,
     \end{equation}
where $\varkappa>0$. In comparison with point \ref{possibility-Martin-boundary}, this way presents the drawback of giving at most one harmonic function.

\subsubsection{} \label{possibility-Varo} Another possibility consists in comparing the solutions to the discrete and the classical (continuous) Dirichlet problems, and to use the numerous results which exist concerning the latter. Unfortunately, this comparison only exists for bounded domains (among which truncated cones), see \cite{Va}. We also refer to \cite{Va1,Va2,Va3} for a study of some properties of discrete (sub)harmonic functions, which are constructed from the continuous analogues.

\subsubsection{} \label{possibility-functional-equation} A last and new approach consists in showing that the generating function of the values taken by the harmonic functions satisfies some functional equation, and in solving the latter. While the functional equation is easily obtained (it turns out to be a rewriting of the Dirichlet problem of point \ref{possibilit-discrete-laplacian}), its resolution will be more arduous!

\smallskip

\smallskip

\smallskip
     
We now briefly present the approach of point \ref{possibility-functional-equation} in the one-dimensional case. Introduce $X = (X(n))_{n\in{\bf Z}_+}$ a random walk on ${\bf Z}_+$, with jumps $\{p_{-1},p_{1}\}$ to the two nearest neighbors, and killed at the boundary of ${\bf Z}_+$ (here the boundary is simply the origin $\{0\}$), see Figure~\ref{fig:dim1}. 
\unitlength=0.6cm
\begin{figure}[t!]
  \begin{center}
    \hspace{-0.6cm}
    \begin{picture}(4.42,2.5)
    \thicklines
    \put(-4.5,1){\vector(1,0){14}}
    \put(-4.5,1){\circle*{0.3}}
    \put(-2.5,1){\circle*{0.3}}
    \put(-0.5,1){\circle*{0.3}}
    \put(1.5,1){\circle*{0.3}}
    \put(3.5,1){\circle*{0.3}}
    \put(5.5,1){\circle*{0.3}}
    \put(7.5,1){\circle*{0.3}}
    \put(-4.65,0.){$0$}
    \put(-1.1,0.){$i-1$}
    \put(1.4,0.){$i$}
    \put(2.9,0.){$i+1$}
    \put(9.15,0.){${\bf Z}_+$}
    \thinlines
    \put(2,1.5){\vector(1,0){1.5}}
    \put(2,1.5){\vector(-1,0){2.5}}
    \put(1.5,1.5){\circle*{0.15}}
    \put(0.2,1.9){$p_{-1}$}
    \put(2.1,1.9){$p_1$}
    \put(-4.5,1.5){\vector(1,0){2}}
    \put(-4.5,1.5){\circle*{0.15}}
    \put(-3.7,1.9){$0$}
  \end{picture}
\end{center}
    \vspace{-2mm}
    \caption{Transition probabilities of the discrete random walk in the one-dimensional case}
    \label{fig:dim1}
\end{figure}
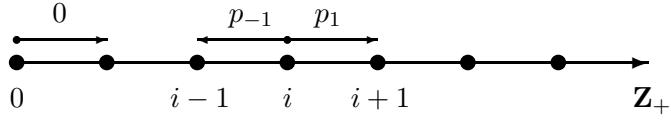
Let $f$ be a discrete harmonic function for this process. We have $f(0)=0$ and for $i_0\geq 1$, the recurrence relation
\begin{equation}
\label{eq:one-dimensional-harmonicity}
     f(i_0) = p_1 f(i_0+1)+p_{-1}f(i_0-1). 
\end{equation}
Let $H(x) =\sum_{i_0\geq 1}f(i_0)x^{i_0-1}$. Multiplying \eqref{eq:one-dimensional-harmonicity} by $x^{i_0-1}$ and summing w.r.t.\ $i_0\geq 1$ yields
\begin{equation*}
     (x^2 p_{-1} - x +p_1)H(x) = p_1 f(1).
\end{equation*}
It is then obvious to deduce an expression of $H(x)$, and also of $f(i_0)$ for any $i_0\geq 1$, by expanding $H(x)$ in power series, or via Cauchy's formul\ae. 
%

\subsection*{Results}
This work aims at presenting the approach summarized in point \ref{possibility-functional-equation} (i.e., finding discrete harmonic functions by solving functional equations) for the walks with small steps (i.e., with jumps only to the eight nearest neighbors) in the quarter plane ${\bf Z}_+^2$, see Figure~\ref{WWSS}.
Our motivations to restrict ourselves to this class are the following: First, it is of special interest in probability \cite{Bi1,CB,FI,FIM,FR2,IRL,KuRa,RaSp4} and in combinatorics as well \cite{BMM,Ra}. Second, our way to solve the functional equation we shall obtain for (the generating function of) the harmonic functions needs complex analysis, and thus dimension $2$. The hypothesis on the small steps is less crucial, and will be commented in Section \ref{Extensions}. Further, we shall mostly assume that the increments of the random walk have zero mean, as it is the most interesting case (see above). This article is hinged on the following main results:
\setcounter{subsubsection}{0}
\subsubsection{}
 \label{result-explicit} We first state a functional equation satisfied by the generating function of the harmonic function (Section \ref{explicit}). We solve it and we show that the solutions are closely related to a certain conformal mapping (Section \ref{solving}). 
     
\subsubsection{} \label{result-MB} We then prove that for the class of zero-mean random walks with small steps in the quarter plane, there is exactly one harmonic function (Section \ref{MB}). 
    
\subsubsection{}\label{result-DW} We also obtain the following two important corollaries (Section \ref{MB}): in \cite{DW}, it is proved that for the zero-mean random walks in the quarter plane, Equation \eqref{eq:asymptotic-DW} holds, where $V(x)$ is shown to be harmonic, but is not made explicit. Here we find an expression for $V(x)$, as a Taylor coefficient of a (relatively) simple function involving sine and arcsine functions.  The second consequence is about lattice path enumeration, and more precisely about the counting of walks with small steps in the quarter plane. Recently, in \cite[Section 1.5]{DW}, the asymptotics of the number of excursions (walks starting and ending at given points) was obtained, up to some multiplicative constant involving two harmonic functions. The analysis we lead here allows us to obtain explicit expressions for both these functions.

\subsubsection{} We relate discrete harmonic functions to their continuous analogues and to the group of the walk, a notion introduced by Malyshev in \cite{MAL} (Section \ref{sec:further_properties}).
    
\subsubsection{}\label{result-extensions} We present possible extensions of our approach and results (Section \ref{Extensions}). Firstly, we show that our methods do work for random walks with non-zero mean increments. In particular, we obtain expressions for the harmonic functions in that case (this improves some results in \cite{IRL,KuRa}). Regarding the non-zero drift case, we also prove that all harmonic functions converge to the unique harmonic function, as the drift goes to $0$. Further, we see that our approach gives expressions for  $t$-harmonic functions (i.e., functions $f$ such that $Pf = tf$ within the cone, $P$ being the generator of the random walk and $t>0$). 

\subsubsection{} In Appendix \ref{sec:sandro} we are interested in the approach in the continuous case: namely, we prove a functional equation for the Laplace transform of the classical harmonic function, and we show how it provides the well-known expression for the harmonic function.

\section{Functional equations for harmonic functions}
\setcounter{equation}{0}
\label{explicit}

\subsection{Harmonic functions for walks with small steps in the quarter plane}

Denote by $(X,Y)=(X(n),Y(n))_{n\in{\bf Z}_+}$ a random walk in the quarter plane ${\bf Z}_+^2$, and let ${\bf P}_{(i_{0},j_{0})}[\mathscr{E}]$ be the probability of event $\mathscr{E}$ conditional on $(X(0),Y(0))=(i_{0},j_{0})$. Throughout we shall make the following assumptions:
\begin{enumerate}[label=(H\arabic{*}),ref={\rm (H\arabic{*})}]
     \item \label{small_jumps} The walk is homogeneous inside of the quarter plane, with transition probabilities $\{p_{i,j}\}_{-1\leq i,j\leq 1}$ to the eight nearest neighbors (we further assume that $p_{0,0}=0$), see Figure \ref{WWSS};
     \item \label{non_degenerate} In the list $p_{1,1},p_{1,0},p_{1,-1},p_{0,-1},p_{-1,-1},p_{-1,0},p_{-1,1},p_{0,1}$,
                 there are no three~consecutive zeros;
     \item \label{drift} The drifts are zero: $\sum_{-1\leq i,j\leq 1}i p_{i,j} = 0$ and $\sum_{-1\leq i,j\leq 1}j p_{i,j} =  0$.
\end{enumerate}
\unitlength=1.1cm
\begin{figure}[t]
        \begin{picture}(6,5.5)
    \thicklines
    \put(1,1){{\vector(1,0){4.5}}}
    \put(1,1){\vector(0,1){4.5}}
    \thinlines
    \put(3,3){\vector(1,1){1}}
    \put(3,3){\vector(-1,-1){1}}
    \put(3,3){\vector(1,0){1}}
    \put(3,3){\vector(-1,0){1}}
    \put(3,3){\vector(0,1){1}}
    \put(3,3){\vector(0,-1){1}}
    \put(3,3){\vector(-1,1){1}}
    \put(3,3){\vector(1,-1){1}}
    \put(4.05,4.2){$p_{1,1}$}
    \put(4.05,3.2){$p_{1,0}$}
    \put(4.05,2.2){$p_{1,-1}$}
    \put(1.05,3.7){$p_{-1,1}$}
    \put(1.05,2.7){$p_{-1,0}$}
    \put(1.05,1.7){$p_{-1,-1}$}
    \put(3.05,1.7){$p_{0,-1}$}
    \put(2.45,4.2){$p_{0,1}$}
    \linethickness{0.1mm}
    \put(1,2){\dottedline{0.1}(0,0)(4.5,0)}
    \put(1,3){\dottedline{0.1}(0,0)(4.5,0)}
    \put(1,4){\dottedline{0.1}(0,0)(4.5,0)}
    \put(1,5){\dottedline{0.1}(0,0)(4.5,0)}
    \put(2,1){\dottedline{0.1}(0,0)(0,4.5)}
    \put(3,1){\dottedline{0.1}(0,0)(0,4.5)}
    \put(4,1){\dottedline{0.1}(0,0)(0,4.5)}
    \put(5,1){\dottedline{0.1}(0,0)(0,4.5)}
\end{picture}
  \vspace{-10mm}
\caption{Walks with small steps in the quarter plane}
\label{WWSS}
\end{figure}
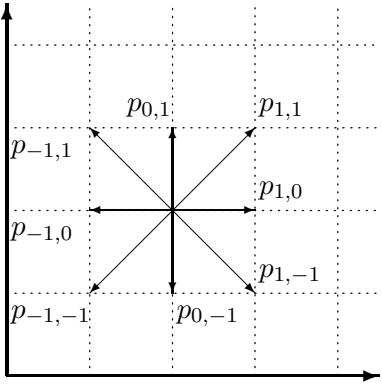
With assumption \ref{small_jumps} we can use the general framework for random walks in the quarter plane developed by Fayolle, Iasnogorodski and Malyshev \cite{FIM}. Assumption \ref{non_degenerate} excludes degenerate random walks, which can typically be analyzed using easier methods. We make assumption \ref{drift} because this is the most interesting case, as seen in the introduction. Moreover, assumption \ref{drift} guarantees that the random walk will hit the boundary with probability one \cite{FIM}. 

In this work we are interested in functions $f$ which are discrete harmonic for these random walks, i.e., in functions $f=(f(i_0,j_0))_{(i_0,j_0)\in{\bf Z}_+^2}$ such that:
\begin{enumerate}[label=(P\arabic{*}),ref={\rm (P\arabic{*})}]
     \item\label{property_harmonicity}For any $i_0,j_0\geq 1$, $
     f(i_0,j_0) = \textstyle \sum_{-1\leq i,j\leq 1}p_{i,j}f(i_0+i,j_0+j)$.
\end{enumerate}
More specifically, we assume that:
\begin{enumerate}[label=(P\arabic{*}),ref={\rm (P\arabic{*})}]
\setcounter{enumi}{1}
     \item\label{property_zero}If $i_0=0$ or $j_0=0$, then $f(i_0,j_0)=0$;
     \item\label{property_positive}If $i_0,j_0>0$, then $f(i_0,j_0)>0$.
\end{enumerate}
A harmonic function $f$ satisfying \ref{property_zero} and \ref{property_positive} is positive harmonic for the random walk killed at the boundary of the quarter plane. In words, denoting by 
\begin{equation*}
     \tau=\inf\{ n\geq 1 : X(n)\leq 0\ \text{or}\ Y(n)\leq 0\} 
\end{equation*}
the first exit time of the random walk from the interior of the quarter plane, one has that
\begin{equation*}
     f(i_0,j_0) ={\bf E}_{(i_0,j_0)}[f(X(1),Y(1)),\tau>1],\qquad \forall i_0,j_0>0.
\end{equation*}

\subsection{Structure of the remainder of Section \ref{explicit}}

We fix a harmonic function $f$ satisfying \ref{property_harmonicity}, \ref{property_zero} and \ref{property_positive}. We first obtain, in Section \ref{AFFE}, a functional equation for the generating function
\begin{equation}
\label{eq:generating_functions_harmonic_functions}
     H(x,y)= \textstyle \sum_{i_0,j_0\geq 1} f(i_0,j_0) x^{i_0-1}y^{j_0-1}.
\end{equation}
This functional equation is stated in \eqref{eq:functional_equation}. It involves a certain bivariate polynomial $L(x,y)$, see \eqref{eq:def_L}, that we study in Section \ref{Notations}. We then state and prove, in Sections \ref{ASFE} and \ref{subsec:proof_lem:BPV}, that both generating functions
\begin{equation}
\label{eq:generating_functions_harmonic_functions_uni}
     H(x,0)=\textstyle \sum_{i_0\geq 1} f(i_0,1) x^{i_0-1},
     \qquad H(0,y)= \sum_{j_0\geq 1} f(1,j_0) y^{j_0-1},
\end{equation}
satisfy certain (rather simple) boundary value problems\footnote{Generally speaking, these are problems of finding an analytic function in a certain domain from a given relation between the boundary values of its real and its imaginary part.}. Finally, starting from the latter, we obtain explicit expressions of $H(x,0)$ and $H(0,y)$, then of $H(x,y)$ via \eqref{eq:functional_equation}, and finally of $f(i_0,j_0)$, for any $i_0,j_0\geq 1$, via the classical Cauchy's formul\ae:
\begin{equation}
\label{eq:Cauchy}
     f(i_0,j_0)=\frac{1}{(2\pi i)^2}\iint \frac{H(x,y)}{x^{i_0}y^{j_0}}\text{d}x\text{d}y,
\end{equation}
where the domain of integration is $\{x\in{\bf C} : \vert x\vert =\epsilon\}\times \{y\in{\bf C} : \vert y\vert =\epsilon\}$, for any $\epsilon\in[0,1)$.

\subsection{A first functional equation, and the exponential growth rate of harmonic functions}
\label{AFFE}

The following polynomial---also called the kernel of the random walk---will be of the highest importance:
     \begin{equation}
     \label{eq:def_L}
          L(x,y)= x y[ \textstyle\sum_{-1\leq i,j\leq 1}
          p_{i,j }x^{-i} y^{-j}  -1].
     \end{equation}
We note that the polynomial $L(x,y)$ is related in a simple way to the transition probabilities of the random walk under consideration.     
\begin{lem}
\label{lem:fe}
For any random walk with property \ref{small_jumps} {\rm(}but not necessarily \ref{non_degenerate} and \ref{drift}{\rm)}, the generating function $H(x,y)$ defined in \eqref{eq:generating_functions_harmonic_functions} satisfies the functional equation
\begin{equation}
\label{eq:functional_equation}
     L(x,y)H(x,y) = L(x,0) H(x,0) + L(0,y) H(0,y)-L(0,0) H(0,0).
\end{equation}
\end{lem}

\begin{proof}
The proof of \eqref{eq:functional_equation} (and thus of Lemma \ref{lem:fe}) is a very simple consequence of \ref{property_harmonicity} and \ref{property_zero}. 
\end{proof}
Though the functional equation \eqref{eq:functional_equation} is just a rewriting of the recurrence relations \ref{property_harmonicity}, it is a crucial step in the way to determine the harmonic functions. We now make a series of remarks on this functional equation.

\subsubsection*{Example 1 (the simple random walk)}
Let us verify \eqref{eq:functional_equation} for the simple random walk, i.e., for the model with transition probabilities $p_{0,1} = p_{1,0} = p_{0,-1} = p_{-1,0} = 1/4$, see Figure \ref{SRW}. In that case, it is well known (see, e.g., \cite{PW}) that there exists a unique positive harmonic function (up to multiplicative constants), with the product form $f(i_0,j_0) = i_0 j_0$. With \eqref{eq:generating_functions_harmonic_functions} and \eqref{eq:generating_functions_harmonic_functions_uni}, we then obtain 
\begin{equation}
\label{eq:expression_HHH_SRW}
     H(x,0) = \frac{1}{(1-x)^2},
     \qquad H(0,y) = \frac{1}{(1-y)^2},
     \qquad H(x,y) = \frac{1}{(1-x)^2}\frac{1}{(1-y)^2}.
\end{equation}
Furthermore, $L(x,y) =(y/4)(x-1)^2+(x/4)(y-1)^2$, $L(x,0) = x/4$, $L(0,y) = y/4$ and $L(0,0) = 0$, see \eqref{eq:def_L}, and it becomes straightforward to verify that \eqref{eq:functional_equation} does hold.
\unitlength=1.1cm
\begin{figure}[t]
        \begin{picture}(6,5.5)
    \thicklines
    \put(1,1){{\vector(1,0){4.5}}}
    \put(1,1){\vector(0,1){4.5}}
    \thinlines
    \put(3,3){\vector(1,0){1}}
    \put(3,3){\vector(-1,0){1}}
    \put(3,3){\vector(0,1){1}}
    \put(3,3){\vector(0,-1){1}}
    \put(4.05,3.15){$1/4$}
    \put(1.45,2.7){$1/4$}
    \put(3.05,1.7){$1/4$}
    \put(2.45,4.15){$1/4$}
    \linethickness{0.1mm}
    \put(1,2){\dottedline{0.1}(0,0)(4.5,0)}
    \put(1,3){\dottedline{0.1}(0,0)(4.5,0)}
    \put(1,4){\dottedline{0.1}(0,0)(4.5,0)}
    \put(1,5){\dottedline{0.1}(0,0)(4.5,0)}
    \put(2,1){\dottedline{0.1}(0,0)(0,4.5)}
    \put(3,1){\dottedline{0.1}(0,0)(0,4.5)}
    \put(4,1){\dottedline{0.1}(0,0)(0,4.5)}
    \put(5,1){\dottedline{0.1}(0,0)(0,4.5)}
\end{picture}
  \vspace{-10mm}
\caption{The simple random walk in the quarter plane}
\label{SRW}
\end{figure}
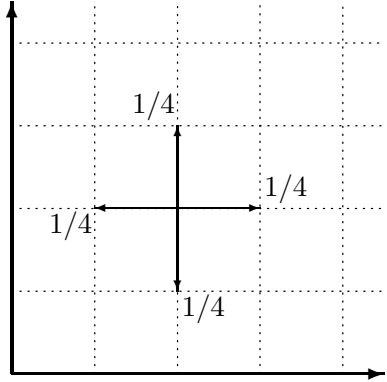

\subsubsection*{Around the functional equation \eqref{eq:functional_equation}}
In the literature, there actually exist many papers on problems that can be reduced to solving functional equations close to \eqref{eq:functional_equation}. This arises in probability theory \cite{FI,FIM,RaSp4}, in queueing theory \cite{CB}, in combinatorics \cite{BMM,Ra}, etc. The functional equations typically have the form
\begin{equation}
\label{eq_fonc_general}
     K(x,y)Q(x,y)=k(x,y)q(x)+\widetilde{k}(x,y)\widetilde{q}(y)+k_{0}(x,y)q_{(0,0)}+\kappa(x,y),
\end{equation}
where $K(x,y)$, $k(x,y)$, $\widetilde k(x,y)$, $k_0(x,y)$ and $\kappa(x,y)$ are known (they are simple functions related to the model, e.g., via the transition probabilities), while the functions $Q(x,y)$, $q(x)$, $\widetilde q(y)$ and $q_{(0,0)}$ are unknown (they may represent the generating functions of Green functions, absorption probabilities, stationary probabilities, counting numbers, etc.). 

Our functional equation \eqref{eq:functional_equation} has three advantages in comparison with the general equation \eqref{eq_fonc_general}:
\begin{itemize}
     \item[\textcolor{green}{\CheckmarkBold}]The coefficient $L(x,0)$ (resp.\ $L(0,y)$, $L(0,0)$) in front of the unknown $H(x,0)$ (resp.\ $H(0,y)$, $H(0,0)$) depends on at most one of the two variables $x$ and $y$;
     \item[\textcolor{green}{\CheckmarkBold}]In the RHS of \eqref{eq:functional_equation}, there is no non-homogeneous term $\kappa(x,y)$ depending on $x$ and $y$;
     \item[\textcolor{green}{\CheckmarkBold}]The unknown functions $q(x)$, $\widetilde q(y)$ and $q_{(0,0)}$ can be expressed in terms of the single unknown $Q(x,y)$, as indeed $q(x)=Q(x,0)$, $\widetilde q(y)=Q(0,y)$ and $q_{(0,0)}=Q(0,0)$.
\end{itemize}
For these reasons, the analytic description of the solutions to the functional equation \eqref{eq:functional_equation} we shall propose here will be more precise than in other works. On the other hand, our particular situation also presents a disadvantage: 
\begin{itemize}
     \item[\textcolor{red}{\XSolidBold}] We shall prove that the unknown functions $H(x,0)$ and $H(0,y)$ satisfy certain boundary value problems (see Lemma \ref{lem:BPV} below). It turns out that these functions will have a singularity on the set defining the boundary condition, which will make the analysis complicated, since we will have to estimate a priori the kind and the order of this singularity.
\end{itemize}

\subsubsection*{Domain of validity of \eqref{eq:functional_equation}}
Equation \eqref{eq:functional_equation} holds (at least) on the domain $\{(x,y)\in{\bf C}^2 : \vert x\vert< 1,\,\vert y\vert < 1\}$. Indeed, we have:

\begin{lem}
\label{lemma:radius_convergence_zero_drift}
For a random walk satisfying \ref{small_jumps}, \ref{non_degenerate} and \ref{drift}, the radius of convergence of $H(x,0)$ and $H(0,y)$ is equal to or larger than one. 
\end{lem}

The proof of Lemma \ref{lemma:radius_convergence_zero_drift} we shall give here\footnote{We are grateful to I.\ Ignatiuk-Robert for suggesting us the ideas of this proof.} uses large deviation theory. It is worth noting that it is the only place in this article where a result is not derived as a consequence of the functional equation \eqref{eq:functional_equation} (and it is an open problem to decide whether it is possible to do so).

\begin{proof}[Proof of Lemma \ref{lemma:radius_convergence_zero_drift}]
By symmetry, it is enough to consider only $H(x,0)$. We prove that its radius of convergence is at least $1$. To that purpose, we show that 
\begin{equation*}
     \limsup_{i\to\infty} \frac{1}{i}\log h({i,1})\leq 0. 
\end{equation*}

Let us define the first hitting time of state $(i_1,j_1)$
\begin{equation*}
     T_{i_1,j_1}=\inf\{n\in {\bf Z}_+ : (X(n),Y(n))=(i_1,j_1)\}, 
\end{equation*}
as well as the Green functions 
\begin{equation*}
     G^{i_0,j_0}_{i_1,j_1} = \textstyle\sum_{n\in {\bf Z}_+} {\bf P}_{(i_0,j_0)}[(X(n),Y(n))=(i_1,j_1)]. 
\end{equation*}
The classical Harnack inequality yields
\begin{equation*}
     \frac{f(i,1)}{f(1,1)}\leq \frac{1}{{\bf P}_{(1,1)}[T_{i,1}<\infty]} = \frac{G^{i,1}_{i,1} }{G^{1,1}_{i,1}}.
\end{equation*}
First, we have that $\sup_{i\geq 1}G^{i,1}_{i,1}<\infty$. Indeed, the Green functions $G^{i,1}_{i,1}$ for the quadrant ${\bf Z}_+^2$ are smaller than the Green functions $\widetilde G^{i,1}_{i,1}$ for the half plane ${\bf Z}\times {\bf Z}_+$. The latter domain being invariant by horizontal translations, the Green functions $\widetilde G^{i,1}_{i,1}$ do not depend on $i$. Accordingly, we have (see \cite{IR0} for the last inequality)
\begin{equation*}
     \sup_{i\geq 1}G^{i,1}_{i,1}\leq \sup_{i\geq 1}\widetilde G^{i,1}_{i,1}=\widetilde G^{1,1}_{1,1}<\infty.
\end{equation*}
To conclude, it is enough to prove that
\begin{equation}
\label{eq:tptptp}
     \liminf_{i\to\infty} \frac{1}{i}\log G^{1,1}_{i,1}\geq 0.
\end{equation}
To that purpose, let us first use the last inequality in \cite[Proof of Proposition 4.2]{IR0} (which is valid not only for the half plane---the topic of \cite{IR0}---but also for the quarter plane); we obtain
\begin{equation}
\label{eq:less_powerful}
     \liminf_{i\to\infty} \frac{1}{i}\log G^{1,1}_{i,1}\geq -\inf _{\phi}I_{[0,T]}(\phi),
\end{equation}
where 
\begin{itemize}
     \item the infimum in the RHS is taken over all trajectories $\phi:[0,T]\to {\bf R}_+^2$ such that $\phi(0)=(0,0)$ and $\phi(T)=(1,0)$;
     \item the rate function $I_{[0,T]}(\phi)$ equals $\int_0^T L(\phi(t))\text{d}t$, with
\begin{equation*}
\left\{\begin{array}{rl}
     L(v)\hspace{-3mm}&=(\log R)^*(v)=\sup_{\alpha \in {\bf R}^2}( \langle \alpha,v\rangle-\log R(\alpha)),\\    
     R(\alpha)\hspace{-3mm}&=\textstyle \sum_{-1\leq i,j\leq 1}p_{i,j}\exp(\langle \alpha,(i,j)\rangle).
\end{array}\right.     
\end{equation*} 
\end{itemize}
Obviously, the inequality \eqref{eq:less_powerful} still holds if one takes the infimum over all linear paths $\phi$, i.e., $\phi(t)=(t/T,0)$ for $t\in[0,T]$. For such functions $\phi$, we have 
\begin{equation*}
     \inf_\phi I_{[0,T]}(\phi) =  T\cdot (\log R )^*(1/T,0).
\end{equation*}
Then, since
\begin{equation*}
     \inf_{T>0} \{T\cdot(\log R )^*(1/T,0)\} =0
\end{equation*}
(see \cite[Page 35]{RO}), we reach the conclusion that \eqref{eq:tptptp} holds.
The proof of Lemma \ref{lemma:radius_convergence_zero_drift} is completed.
\end{proof}

We shall prove in Corollary \ref{cor:exactly1} that for any (non-zero) harmonic function $f(i_0,j_0)$, one has $H(1,0)=H(0,1)=\infty$. Together with Lemma \ref{lemma:radius_convergence_zero_drift}, this entails that the radius of convergence of $H(x,0)$ and $H(0,y)$ is exactly one. As a consequence, the harmonic function $f(i_0,j_0)$ goes to infinity at a subexponential rate as $i_0,j_0$ go to infinity. 

The latter facts (namely, that there is only one possible exponential growth rate for the harmonic functions, and that this rate is one) is very specific to the zero drift case. For a random walk having a non-zero drift (and satisfying \ref{small_jumps} and \ref{non_degenerate}), there will exist infinitely many possible exponential growth rates for the harmonic functions. It is precisely this phenomenon which will eventually entail that the Martin boundary of a random walk with non-zero drift is composed of infinitely many harmonic functions \cite{IRL,KuRa}; see Section~\ref{subsec:SRWNZD} for more details.

\subsection{A second functional equation}
\label{ASFE}

Our aim now is to prove, starting from the first functional equation \eqref{eq:functional_equation}, that there exists a domain $\mathscr{G}_X$ (as on Figure \ref{fig:ex_curve}) in which the unknown function $H(x,0)$ satisfies a boundary value problem, as follows:

\unitlength=0.6cm
\begin{figure}[t]
\begin{center}
    \begin{picture}(0,5)
    \thicklines
    \put(-5,0){\vector(1,0){10}}
    \put(0,-5){\vector(0,1){10}}
       {\psarc[linecolor=red](2.1,0){-0.5}{-63}{0}}
    {\put(3.5,0){\textcolor{black}{\circle*{0.3}}}}
    \put(3.6,0.4){$1$}
    {\put(0.5,0){\textcolor{black}{\circle*{0.3}}}}
    \put(0.3,0.4){$x_1$}
    \put(2.3,0.4){\textcolor{red}{$\theta$}}
    \put(-3.3,0.4){\textcolor{blue}{$\mathscr{G}_X$}}
   \textcolor{blue}{\qbezier(3.5,0)(2,3.5)(0,3.5)
    \qbezier(3.5,0)(2,-3.5)(0,-3.5)
    \qbezier(0,3.5)(-3.5,3.5)(-3.5,0)
    \qbezier(-3.5,0)(-3.5,-3.5)(0,-3.5)}
     \end{picture}
\end{center}
\vspace{25mm}
\caption{The domain $\mathscr{G}_X$ is simply connected and symmetrical w.r.t.\ the horizontal axis. It is smooth everywhere except at $1$, where it may have a corner point}
\label{fig:ex_curve}
\end{figure}
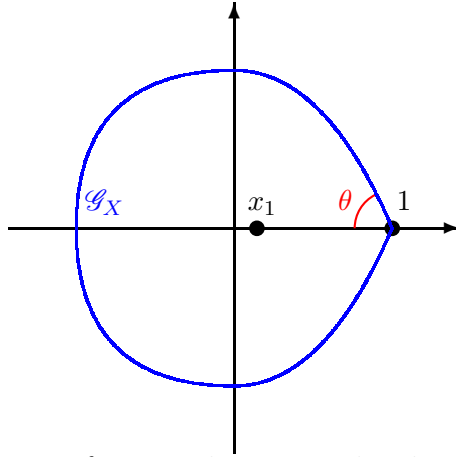

\begin{lem}
\label{lem:BPV}
 $H(x,0)$ satisfies the following boundary value problem {\rm(}below and throughout, $\overline{x}$ stands for the conjugate value of $x\in{\bf C}${\rm)}:
     \begin{enumerate}[label={\rm(\roman{*})},ref={\rm(\roman{*})}]
          \item\label{class_function} $H(x,0)$ is analytic in $\mathscr{G}_X$;
          \item\label{class_function_2} $H(x,0)$ is continuous on $\overline{\mathscr{G}_X}\setminus \{1\}$;
          \item\label{boundary_condition} For all $x$ on the boundary of $\mathscr{G}_X$ except at $1$,
          $
               L(x,0)H(x,0)-L(\overline{x},0)H(\overline{x},0)=0.
          $
     \end{enumerate}
\end{lem}     
The problem of finding functions satisfying \ref{class_function}, \ref{class_function_2} and \ref{boundary_condition} is a particular instance of a boundary value problem with shift (the complex conjugation plays in \ref{boundary_condition} the role of the shift), see \cite{LIT} for an extensive treatment of this topic. 

We postpone the proof of Lemma \ref{lem:BPV} to Section \ref{subsec:proof_lem:BPV}, since before we need to introduce some basic properties of the kernel $L(x,y)$, so as in particular to define properly the domain $\mathscr{G}_X$.

\subsection{Basic properties of the kernel}
\label{Notations}
This part is purely technical, and presents some properties of the kernel $L(x,y)$ introduced in \eqref{eq:def_L}, that are crucial for the proofs.  

\subsubsection*{Branch points}
The kernel $L(x,y)$ can also be written as\footnote{For readers used to \cite{FIM}, we notice that the notations we take here are not the usual ones, as the kernel $L(x,y)$ in \eqref{eq:def_L} is the reciprocal one of the kernel $Q(x,y)$ in \cite{FIM}: $L(x,y)=x^2y^2Q(1/x,1/y)$.}
     \begin{equation}
     \label{eq:alternative_definition_L}
          L(x,y) = \alpha (x) y^{2}+ \beta (x) y + \gamma (x) = \widetilde{\alpha }(y) x^{2}+
          \widetilde{\beta }(y) x + \widetilde{\gamma }(y),
     \end{equation}
where
          \begin{equation}
           \label{def_a_b_c}
          \left\{\begin{array}{rll}
\alpha (x)\hspace{-2mm} &=&\hspace{-2mm}  p_{-1,-1}x^{2}+ p_{0,-1}x+ p_{1,-1},\\
                \beta(x) \hspace{-2mm} &=&\hspace{-2mm} p_{-1,0}x^{2}- x+ p_{1,0},\\
               \gamma(x)\hspace{-2mm}  &=&\hspace{-2mm} p_{-1,1}x^{2}+ p_{0,1}x+ p_{1,1},\\
               \widetilde{\alpha }(y)\hspace{-2mm}  &=& \hspace{-2mm} p_{-1,-1}y^{2}+ p_{-1,0}y+ p_{-1,1},\\
               \widetilde{\beta}(y)\hspace{-2mm}  &=&\hspace{-2mm} p_{0,-1}y^2- y+ p_{0,1},\\
               \widetilde{\gamma}(y) \hspace{-2mm} &=&\hspace{-2mm} p_{1,-1}y^{2}+ p_{1,0}y+ p_{1,1}.\end{array}\right.
     \end{equation}
We also define
     \begin{equation}
     \label{def_d}
          \delta (x)=\beta (x)^{2}-4\alpha (x) \gamma(x),
          \qquad \widetilde{\delta }(y)=
          \widetilde{\beta }(y)^{2}-4
          \widetilde{\alpha }(y)\widetilde{\gamma }(y),
     \end{equation}
which are the discriminants of the polynomial $L(x,y)$ as a function of $y$ and $x$, respectively. The roots of these polynomials are called the branch points of the kernel. The following facts regarding the polynomial $\delta$ are proved in \cite[Chapter 2]{FIM}: Under \ref{small_jumps}, \ref{non_degenerate} and \ref{drift}, $\delta $ has degree three or four. Further, $1$ is a root of order two of $\delta$. We denote the roots of $\delta$ by  $\{x_\ell\}_{1\leq \ell\leq 4}$, with
     \begin{equation}
     \label{properties_branch_points}
          |x_1|\leq 1 \leq |x_4|,
          \qquad
          x_2 = x_3 = 1,
     \end{equation}
and $x_4=\infty$ if $\delta $ has degree three. We have $x_1\in[-1,1)$ and $x_4\in(1,\infty) \cup \{\infty\}\cup (-\infty,-1]$. Further, on ${\bf R}$, $\delta(x)$ is negative if and only if $x\in [x_1,x_4]\setminus \{1\}$. 
The polynomial $\widetilde \delta$ in \eqref{def_d} and its roots $\{y_\ell\}_{1\leq \ell\leq 4}$ satisfy similar properties, see their location on Figure \ref{fig:ex_branch_points}. 

\subsubsection*{Algebraic functions defined by the kernel}
In what follows, we denote by $X(y)$ and $Y(x)$ the algebraic functions defined by $L(X(y),y)=0$ and $L(x,Y(x))=0$. With \eqref{eq:alternative_definition_L} and \eqref{def_d} we have
     \begin{equation}
     \label{expression_X_Y}
          X(y)=\frac{-\widetilde \beta(y)\pm \widetilde \delta(y)^{1/2}}{2 \widetilde \alpha(y)},
          \qquad Y(x)=\frac{-\beta(x)\pm \delta(x)^{1/2}}{2 \alpha(x)}.
     \end{equation}
Functions $X(y)$ and $Y(x)$ both have two branches, called $X_0$, $X_1$ and $Y_0$, $Y_1$, which are meromorphic on ${\bf C}\setminus [y_1,y_4]$ and ${\bf C}\setminus [x_1,x_4]$, respectively, see Figure \ref{fig:ex_branch_points}. We fix the notations by requiring that on the whole of ${\bf C}$,
     \begin{equation*}
          |X_0(y)|\leq |X_1(y)|, \qquad |Y_0(x)|\leq |Y_1(x)|.
     \end{equation*}
     \unitlength=0.6cm
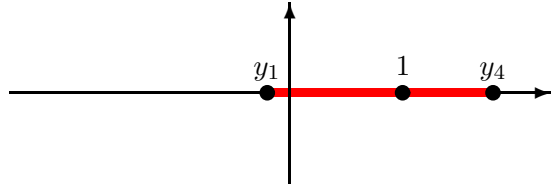
\begin{figure}[t]
\begin{center}
\begin{tabular}{ccccc}
    \begin{picture}(0,2)
    \thicklines
    \put(4.5,0){\vector(1,0){1.5}}
    \put(-6,0){\line(1,0){5.6}}
    \linethickness{1mm}
    \textcolor{red}{\put(-0.5,0){\line(1,0){5}}}
    \thicklines
    \put(0,-2){\vector(0,1){4}}
    \put(-.5,0){\textcolor{black}{\circle*{0.3}}}
    \put(2.5,0){\textcolor{black}{\circle*{0.3}}}
    \put(4.5,0){\textcolor{black}{\circle*{0.3}}}
    \put(-0.8,0.4){$y_1$}
    \put(2.35,0.4){$1$}
    \put(4.2,0.4){$y_4$}
     \end{picture}
    \end{tabular}
\end{center}
\vspace{10mm}
\caption{The functions $X_0(y)$ and $X_1(y)$ are meromorphic functions on the cut plane ${\bf C}\setminus [y_1,y_4]$}
\label{fig:ex_branch_points}
\end{figure}
     
\subsubsection*{On a curve defined by the complex values of function $X(y)$}

To conclude Section \ref{Notations},  we introduce the curve
     \begin{equation*}
          X([y_1,1])=X_0([y_1,1])\cup X_1([y_1,1]),
     \end{equation*}
which is symmetrical w.r.t.\ the real axis (since for $y\in [y_1,1]$, $X_0(y)$ and $X_1(y)$ are complex conjugates), and goes around the point $x_1$ (see \cite[Theorem 5.3.3]{FIM}). 

The curve $X([y_1,1])$ is smooth except at $X(1) = 1$, where it may have a corner point. More specifically, at $1$ the angle between the segment $[x_1,1]$ and the curve is given by \cite{FR2}
     \begin{equation}
     \label{exp_theta}
          \pm\theta=\pm\arccos \left(-\frac{\sum_{-1\leq i,j\leq 1}i j p_{i,j}}
          {\sqrt{(\sum_{-1\leq i,j\leq 1}i^2 p_{i,j})\cdot(\sum_{-1\leq i,j\leq 1}j^2 p_{i,j})}}\right).
     \end{equation}
We denote by $\mathscr{G}_X$ the open set bounded by the curve $X([y_1,1])$, which in addition contains the point $x_1$ (or equivalently $0$). See Figure \ref{fig:ex_curve} for an example of domain $\mathscr{G}_X$.

\subsubsection*{Example 1 (continued)} For instance, for the simple random walk, $X([y_1,1])$ is the unit circle (see below, or see \cite[Theorem 5.3.3]{FIM} for the original proof), hence $\mathscr{G}_X$ is the unit disc. In that case, $X([y_1,1])$ is smooth at $1$, which is coherent with $\theta = \arccos (0) =\pi/2$ in \eqref{exp_theta}. We now briefly explain how to prove that $\mathscr{G}_X$ is the unit disc. Take $y\in[y_1,1]$. Then $\widetilde\delta(y)\leq 0$, and thus $\widetilde\delta(y)^{1/2}$ is purely imaginary. Then
\begin{equation*}
     \left\vert \frac{-\widetilde \beta(y)\pm \widetilde \delta(y)^{1/2}}{2 \widetilde \alpha(y)}\right\vert^2 =\frac{\widetilde \beta(y)^2-\widetilde \delta(y)}{4\widetilde\alpha(y)^2 }=\frac{\widetilde\gamma(y)}{\widetilde\alpha(y)}=1,
\end{equation*}
where the last equality comes from the fact that for the simple random walk, $\widetilde\gamma(y)=\widetilde\alpha(y)$. We conclude that $X([y_1,1])$ is included in the unit circle. We have the equality as it follows from easy calculations that $X(1)=1$ and $X(y_1)=-1$.

\subsection{Proof of Lemma \ref{lem:BPV}}
\label{subsec:proof_lem:BPV}

Let $\mathscr{G}_X$ be the set introduced above, and let $\overline{\mathscr{G}_X}$ be its closure (in ${\bf C}\cup\{\infty\}$). This part aims at showing Lemma \ref{lem:BPV}.

\begin{proof}[Proof of {\rm \ref{class_function}}]
Note that from Lemma \ref{lemma:radius_convergence_zero_drift}, we already know that $H(x,0)$ is analytic within the open unit disc. Accordingly, if the domain $\mathscr{G}_X$ is included in the unit disc, there is nothing to prove. Note that like the simple random walk (see Section \ref{Notations}), many walks do satisfy this property. 

On the other hand, there also exist walks such that part of $\mathscr{G}_X$ lies outside the unit disc. This is for example the case for the (Gessel's) walk with $p_{1,1} = p_{1,0} = p_{-1,-1} = p_{-1,0} = 1/4$, see \cite{Ra}. In that case, \ref{class_function} follows from an analytic continuation argument based on  \eqref{eq:functional_equation}. This argument is exactly the same as that used in \cite[Section 4]{Ra}, and we refer to this work for the details.
\end{proof}

\begin{proof}[Proof of {\rm \ref{class_function_2}}]
We refer to \cite[Section 4]{Ra}, where a stronger statement (on the analytic continuation of the whole of ${\bf C}\setminus [1,x_4]$) is proved for a quite similar functional equation.
\end{proof}

\begin{proof}[Proof of {\rm \ref{boundary_condition}}]
Proving \ref{boundary_condition} is classical, see, e.g., \cite{FIM,KuRa,Ra}. However, we write down some details here, as this illustrates concrete and standard manipulations on functional equations of the type \eqref{eq:functional_equation} and \eqref{eq_fonc_general}. 

We first evaluate \eqref{eq:functional_equation} both at $X_0(y)$ and $X_1(y)$. At this point, the evaluation of the functional equation is just formal, since a priori this is valid only if $\vert X_0(y)\vert <1$, $\vert X_1(y)\vert <1$ and simultaneously $\vert y\vert <1$ (see Lemma \ref{lemma:radius_convergence_zero_drift}). This way, the kernel $L(x,y)$ vanishes, and the LHS of \eqref{eq:functional_equation} too. We thus have, for $i\in\{0,1\}$,
\begin{equation*}
     0= L(X_i(y),0)H(X_i(y),0)+L(0,y)H(0,y)-L(0,0)H(0,0).
\end{equation*}
Making the difference of the two above equations leads to
\begin{equation*}
     L(X_1(y),0)H(X_1(y),0)-L(X_0(y),0)H(X_0(y),0)=0.
\end{equation*}
The latter equation relates the function $L(x,0)H(x,0)$ evaluated at two different points. Choosing $y\in[y_1,1]$ (the fact that we can do this evaluation is not obvious, but it comes from similar arguments as in \ref{class_function_2}) and noticing that for such values of $y$, $X_0(y)$ and $X_1(y)$ are complex conjugate, we obtain that for all $x\in X([y_1,1])$ except at $1$,
\begin{equation*}
     L(x,0)H(x,0)-L(\overline{x},0)H(\overline{x},0)=0,
\end{equation*}
which is nothing else but \ref{boundary_condition}.
\end{proof}

\subsection{A key lemma} 
We conclude Section \ref{explicit} by stating a lemma (all-present in Section \ref{solving}) on a certain boundary value problem close to that of Section \ref{ASFE}.

\begin{lem}
\label{lem_constant}
Let $h(x)$ be a function satisfying {\rm\ref{class_function}}, {\rm\ref{class_function_2_tilde}} and {\rm\ref{boundary_condition}}, where
\begin{enumerate}[label={\rm($\widetilde{\text{\roman{*}}})$},ref={($\widetilde{\text{\roman{*}}}$)}]
\setcounter{enumi}{1}
     \item \label{class_function_2_tilde} $h(x)$ is continuous on $\overline{\mathscr{G}_X}$.
\end{enumerate}
Then $h(x)$ necessarily is a constant function.
\end{lem}
\begin{proof}
The proof of this result is contained in both \cite{FIM} and \cite{LIT}.
\end{proof}

An immediate consequence of Lemma \ref{lem:BPV} and Lemma \ref{lem_constant} is the following:
\begin{cor}
\label{cor:exactly1}
For any non-zero harmonic function, $H(1,0)=H(0,1)=\infty$.
\end{cor}
Indeed, if for instance $H(1,0)$ were finite, then the function $H(x,0)$ would satisfy the boundary value problem {\rm\ref{class_function}}, {\rm\ref{class_function_2_tilde}} and {\rm\ref{boundary_condition}}, and thus $H(x,0)$ would be a constant function. This would imply that $f(i_0,j_0)$ is the harmonic function identically equal to zero. Further, together with Lemma \ref{lemma:radius_convergence_zero_drift}, Corollary \ref{cor:exactly1} entails that the radius of convergence of $H(x,0)$ and $H(0,y)$ is exactly one.

\section{Solving the functional equations}
\setcounter{equation}{0}
\label{solving}

In this section we characterize the harmonic functions satisfying \ref{property_harmonicity}, \ref{property_zero} and \ref{property_positive}. Precisely, we shall prove the following result: define
     \begin{equation}
     \label{def_CGF_u}
          w(x)=\sin\left(\frac{\pi}{\theta}
          \left[\arcsin(T(x))-\frac{\pi}{2}\right]\right)^2,
     \end{equation}
with $\theta$ as in \eqref{exp_theta}, and with    
     \begin{equation}
     \label{def_f}
          T(x) = \frac{1}{\sqrt{\frac{1}{3}-\frac{2f(x)}{\delta''(1)}}},\qquad f(x)=\left\{\begin{array}{lll}
          \displaystyle \frac{\delta''(x_{4})}{6}+\frac{\delta'(x_{4})}{x-x_{4}}& \text{if} & x_{4}\neq \infty,\vspace{1.5mm}\\
          \displaystyle \frac{\delta''(0)}{6}+\frac{\delta'''(0)x}{6} & \text{if} & x_{4}=\infty.\end{array}\right.
     \end{equation}
Subtracting by $w(0)$, we can assume that $w(0)=0$. Define further
\begin{equation}
\label{eq:mu-nu}
     \nu =-w(X_0(0)),
     \qquad
     \mu =f(1,1)\times \left\{\begin{array}{lll}
     \displaystyle \frac{2p_{-1,1}}{w''(0)}&\text{if}& p_{1,1}=0 \text{ and } p_{0,1}= 0,\vspace{1.5mm}\\
     \displaystyle \frac{p_{0,1}}{w'(0)}&\text{if}& p_{1,1} = 0 \text{ and } p_{0,1}\neq 0,\vspace{1.5mm}\\
     \displaystyle-\frac{p_{1,1}}{w(X_0(0))} &\text{if}& p_{1,1}\neq 0.
     \end{array}\right.
\end{equation}
\begin{thm}
\label{thm:main-harmonic}
Define $w(x)$ as in \eqref{def_CGF_u}, and $\mu$ and $\nu$ as in \eqref{eq:mu-nu}. We have
\begin{equation}
\label{eq:l}
     H(x,0) = \mu\frac{w(x)+\nu}{L(x,0)}.
\end{equation}
\end{thm}

Of course, a similar expression could be written for $H(0,y)$, after a suitable change of the parameters (namely, $p_{i,j}\to p_{j,i}$). As for $H(0,0)$, it is simply equal to $f(1,1)$, see \eqref{eq:generating_functions_harmonic_functions}, and an expression for the complete generating function $H(x,y)$ then follows from \eqref{eq:functional_equation}. Finally, the Taylor coefficients $f(i_0,j_0)$ of $H(x,y)$ can be obtained via Cauchy's formul\ae, see \eqref{eq:Cauchy}.  

The starting point for proving Theorem \ref{thm:main-harmonic} is the boundary value problem that satisfies the function $H(x,0)$ (see Section \ref{ASFE}). In Section \ref{subsec:1}, we show that $H(x,0)$ is necessarily of the form $S(w(x))/L(x,0)$, where $S$ is a rational function\footnote{We would like to insist on the fact that the function $S(w(x))/L(x,0)$ will not be the generating function associated with a harmonic function for any rational function $S$; however, it will solve the problem \ref{class_function}, \ref{class_function_2} and \ref{boundary_condition} for any $S$.}. In Section \ref{subsec:2}, we prove that $S(X)$ is actually a polynomial of the first order, and in Section \ref{subsec:3} we find the expression of this polynomial (i.e., the constants $\mu$ and $\nu$ in \eqref{eq:l}). In Section \ref{mapping} we compute $w(x)$ in terms of sine and arcsine functions.

\subsection{Expression of the solutions to the boundary value problem}
\label{subsec:1}

For a given non-real $z_0$ and a non-zero harmonic function $f(i_0,j_0)$, consider the function
\begin{equation*}
     \mathcal H(x) = \frac{1}{L(x,0)H(x,0) - z_0}.
\end{equation*}
We first notice that for $x$ on $X([y_1,1])$ (the boundary of $\mathscr G_X$), $\mathcal H(x)\neq \infty$. Indeed, remember (see in particular Section \ref{ASFE}) that $L(x,0)H(x,0)$ takes real values on $X([y_1,1])$ and that $z_0$ is non-real. Further, let us denote by $N_p$ the number of poles that $\mathcal H(x)$ has within $\mathscr G_X$. First, by the principle of isolated singularities \cite[Section 3.2]{SG}, we obtain that $N_p<\infty$ (for any non-real $z_0$). Second:
\begin{lem}
\label{lem:N_p-non-zero}
     For any non-real $z_0$, we have $N_p\geq 1$.
\end{lem}

\begin{proof}
Assume that $N_p = 0$. Then $\mathcal H(x)$ satisfies {\rm\ref{class_function}}, {\rm\ref{class_function_2_tilde}} and {\rm\ref{boundary_condition}}, so that by Lemma \ref{lem_constant}, we conclude that $\mathcal H(x)$ is a constant function $C$. Accordingly, $H(x,0)=(z_0+1/C)/L(x,0)$. 
Having $z_0+1/C\neq 0$ is impossible, as $1/L(x,0)$ is either not a power series (if $L(0,0)=0$), or is a power series whose some of the coefficients are negative (if $L(0,0)\neq 0$). It follows that $z_0+1/C = 0$, and thus that $H(x,0)$ is identically zero. Accordingly, $f(i_0,j_0)=0$ for any $i_0,j_0\geq 0$, which is a contradiction with our starting assumption.
\end{proof}

Though this is not crucial for our analysis, it is interesting to note that we can prove, purely analytically, that $N_p$ is intrinsic, in the sense that it does not depend on $z_0$:
\begin{lem}
\label{lem:N_p-constant}
     $N_p$ is independent of the choice of $z_0\in{\bf C}\setminus {\bf R}$.
\end{lem}

\begin{proof}
The number of poles $N_p$ can be calculated as follows (see \cite[Section 3.8.2]{SG}): if the contour $X([y_1,1])$ is traversed in the counter-clockwise sense, and if $N_z$ (resp.\ $N_p$) is the number of zeros (resp.\ poles) of $\mathcal H(x)$ inside $X([y_1,1])$ (they are counted as many times as their order indicates), we have\footnote{A priori, this Cauchy's theorem cannot be applied as such, as the contour $X([y_1,1])$ is not included in a domain where $\mathcal H(x)$ is analytic. However, since $\mathcal H(x)$ can be analytically continued through any point of $X([y_1,1])\setminus \{1\}$ (see \cite[Section 4]{Ra} for a similar statement), and since $\mathcal H(x)\to\infty$ as $x\to 1$, this result still holds.}
\begin{equation}
\label{eq:number-poles}
     \frac{1}{2\pi i}\int_{X([y_1,1])}\frac{\mathcal H'(x)}{\mathcal H(x)}\text{d} x = N_z-N_p = -N_p,
\end{equation}
since $N_z = 0$ (indeed, $H(x,0)\neq \infty$ for $x\in\mathscr G_X$). By a classical reasoning ($N_p$ takes only integer values and is continuous w.r.t.\ $z_0$ in the positive (or negative) half plane), we obtain that $N_p$ is constant for $z_0$ in a given half plane. Since the coefficients of $L(x,0)H(x,0)$ are real, we reach the conclusion that $N_p$ takes the same constant value for $z_0$ in the two half planes, and the proof of Lemma \ref{lem:N_p-constant} is completed.
\end{proof}

We now construct a function $C(x)$ such that $\mathcal H(x)-C(x)$ satisfies {\rm\ref{class_function}}, {\rm\ref{class_function_2_tilde}} and {\rm\ref{boundary_condition}}. With Lemma \ref{lem_constant}, it will then be obvious to deduce that $\mathcal H(x)-C(x)$ is a constant function. Constructing, or even showing the existence of such a function $C(x)$ is not easy, as it has to satisfy many properties: it must be meromorphic within the domain $\mathscr G_X$ (point {\rm\ref{class_function}}), it has to compensate for the poles of $\mathcal H(x)$ without adding any other poles (point {\rm\ref{class_function_2_tilde}}), and simultaneously it has to satisfy $C(\overline{x})=C(x)$ on the boundary of $\mathscr G_X$ (point {\rm\ref{boundary_condition}}). 

We first focus on the point {\rm\ref{class_function_2_tilde}}, and we explain how to cancel the singularities of $\mathcal H(x)$. Let us call $x_1,\ldots,x_{q}$ the poles of $\mathcal H(x)$ in $\mathscr G_X$, with respective multiplicities $n_1,\ldots,n_q$, so that $n_1 + \cdots + n_q = N_p$. Let $\ell\in\{1,\ldots ,q\}$ be fixed. Since the function ${1}/({x - x_\ell})$ has a pole of order $1$ at $x_\ell$ and is analytic elsewhere, we can find a polynomial $P_\ell$ of order $n_\ell$ such that 
\begin{equation*}
     \mathcal H(x) - P_\ell\left(\frac{1}{x -x_\ell}\right)
\end{equation*}
is analytic at $x_\ell$. Accordingly, the function
\begin{equation*}
     \mathcal H(x) - \textstyle \sum_{\ell = 1}^q \displaystyle P_\ell\left(\frac{1}{x - x_\ell}\right)
\end{equation*}
satisfies {\rm\ref{class_function}} and {\rm\ref{class_function_2_tilde}}. However, there is no reason that it takes the same value at $x$ and $\overline{x}$ for $x$ on the boundary (condition {\rm\ref{boundary_condition}}). 

We can adapt the argument above so as to satisfy condition {\rm\ref{boundary_condition}}. Specifically, we shall construct a function $w(x)$ such that eventually, the function
\begin{equation}
\label{eq:important_function}
     \mathcal H(x) - \textstyle \sum_{\ell = 1}^q \displaystyle P_\ell\left(\frac{1}{w(x) - w(x_\ell)}\right)
\end{equation}
will satisfy {\rm\ref{class_function}}, {\rm\ref{class_function_2_tilde}} and {\rm\ref{boundary_condition}}. If it exists, the function $w$ has to satisfy many properties. We first ask that
\begin{enumerate}[label=(W\arabic{*}),ref={\rm (W\arabic{*})}]
\item\label{W-boundary} For all $x$ on the boundary of $\mathscr{G}_X$ where $w(x)$ is (defined and) finite, $w(x)=w(\overline{x})$.
\end{enumerate}
As a consequence of \ref{W-boundary}, the function \eqref{eq:important_function} satisfies {\rm\ref{boundary_condition}} (at least formally, since at this point we do not know that $w(x)$ is finite at any point of the boundary of $\mathscr{G}_X$). In order to obtain {\rm\ref{class_function}}, the above construction (with the function $1/(x-x_\ell)$ instead of $1/(w(x)-w(x_\ell))$) and Equation \eqref{eq:important_function} suggest to impose that the pole that ${1}/({w(x) - w(x_\ell)})$ has at any point $x_\ell$ is of order $1$. We shall therefore ask that
\begin{enumerate}[label=(W\arabic{*}),ref={\rm (W\arabic{*})}]
\setcounter{enumi}{1}
     \item\label{W-conformal} $w$ is injective in $\mathscr{G}_X$;
     \item\label{W-analytic} $w$ is analytic on $\mathscr{G}_X$.
\end{enumerate}
The function $w$ cannot be continuous on $\overline{\mathscr{G}_X}$, for else it would be constant, see Lemma \ref{lem_constant}. It thus takes the value $\infty$ at a point on the boundary of ${\mathscr{G}_X}$. Moreover, it cannot take the value $\infty$ twice, otherwise it would not be injective. Considering instead of $w$ the function $1/(w(x)-w(1))$, we may assume that $w(1)=\infty$. It follows that
\begin{enumerate}[label=(W\arabic{*}),ref={\rm (W\arabic{*})}]
\setcounter{enumi}{3}
\item\label{W-continuous} $w$ is continuous on $\overline{\mathscr{G}_X}\setminus \{1\}$ and $w(1)=\infty$.
\end{enumerate}
The existence of $w$ satisfying \ref{W-boundary}--\ref{W-continuous} follows from general considerations on conformal mappings, see \cite[Chapter 2]{LIT}. We can further assume that $w(x)$ is real for $x$ on the boundary of ${\mathscr{G}_X}$. It is shown in \cite{FR2} that $w(x)$ is the unique function for which \ref{W-boundary}--\ref{W-continuous} hold (up to additive and multiplicative constants). We shall give further properties of the function $w(x)$ in Section \ref{mapping}. In particular, we shall prove that unexpectedly, $w(x)$ can be made explicit in terms of sine and arcsine functions. 

It comes from \eqref{eq:important_function} and Lemma \ref{lem_constant} that the following choice for $C(x)$ is thus suitable:
\begin{equation*}
     C(x) = \textstyle \sum_{\ell = 1}^q \displaystyle P_\ell\left(\frac{1}{w(x) - w(x_\ell)}\right).
\end{equation*}
We conclude that the function $\mathcal H(x)-C(x)$ is constant. Since $w(1)=\infty$ (see \ref{W-continuous}), we have $C(1)=0$ (imposing $P_\ell(0)=0$). Furthermore, we have $H(1,0)=\infty$ (Corollary~\ref{cor:exactly1}), so that $\mathcal H(1)=0$, and the constant function $\mathcal H(x)-C(x)$ is zero. Finally, we have
\begin{equation*}
     H(x,0) = \frac{1}{L(x,0)}\left(z_0+\frac{1}{C(x)}\right).
\end{equation*}

\subsection{Form of the rational function}
\label{subsec:2}

To summarize Section \ref{subsec:1}, given some $z_0\in{\bf C}\setminus {\bf R}$, we have constructed a rational function 
\begin{equation}
\label{eq:rational-R}
     R(X)=\textstyle\sum_{\ell = 1}^q \displaystyle P_\ell\left(\frac{1}{X-w(x_\ell)}\right)
\end{equation}
such that $H(x,0)=(z_0+1/R(w(x)))/L(x,0)$. We now prove that $q=1$ in \eqref{eq:rational-R}, and that the polynomial $P_1$ has order $1$.

\begin{lem}
\label{lem:1/R-polynomial}
The rational function $1/R(X)$ must be a polynomial.
\end{lem}

Before proving Lemma \ref{lem:1/R-polynomial}, note its following consequence:
\begin{cor}
\label{cor:q=1}
We have $q=1$ in \eqref{eq:rational-R} and $P_1(X)=\mu X^{n_1}$, for some $\mu\in{\bf C}$ and $n_1\in{\bf Z}_+$.
\end{cor}

\begin{proof}
It is elementary to prove that for a rational function $R(X)$ as in \eqref{eq:rational-R}, a necessary and sufficient condition for $1/R(X)$ to be a polynomial is that $q=1$ and $P_1(X)=\mu X^{n_1}$ in \eqref{eq:rational-R}.
\end{proof}

\begin{proof}[Proof of Lemma \ref{lem:1/R-polynomial}]
Assume that $1/R(X)$ is not a polynomial. Then there exists $z_1\in{\bf C}$ such that $1/R(z_1)=\infty$. Introduce $t_1\in\overline{\mathscr G_X}$ such that\footnote{If $t_1$ belongs to the (open) set $\mathscr G_X$, then $t_1$ is unique, while if $t_1$ is on the boundary, two values (namely, $t_1$ and $\overline{t_1}$) are possible.} $w(t_1)=z_1$. Note that $t_1\neq 1$, as $w(1)=\infty$. We conclude that $H(t_1,0)=\infty$, which is an apparent contradiction with \ref{class_function} or \ref{class_function_2} of Section \ref{ASFE}.
\end{proof}

So far, we have shown that there exists a constant $\mu$ (equal to $\mu^{-1}$ with the notations of Corollary \ref{cor:q=1}) such that
\begin{equation}
\label{eq:lbo}
     H(x,0) = \frac{z_0+\mu(w(x)-w(x_1))^{n_1}}{L(x,0)}.
\end{equation}

\begin{lem}
     We have $n_1 = 1$ and $\mu\in{\bf R}\setminus \{0\}$ in \eqref{eq:lbo}.
\end{lem}
\begin{proof}
First, we have $\mu\neq 0$ (we already saw that $1/L(x,0)$ cannot be the generating function of positive numbers). As $x\in (0,1)$ goes to $1$, both $w(x)$ and $H(x,0)$ are real and go to $\infty$. For the polynomial $z_0+\mu(X-w(x_1))^{n_1}$ to be real, we must have $\mu\in{\bf R}$. If $n_1\geq 2$, the same argument (together with the binomial formula) gives that $w(x_1)$ and $z_0$ have to be real too. But $z_0\in{\bf C}\setminus {\bf R}$, so that $n_1=1$.
\end{proof}
Finally, we have proved \eqref{eq:l}, and it remains to compute the constants $\mu$ and $\nu$.

\subsection{Computation of the constants}
\label{subsec:3}

This part aims at computing the values of $\mu$ and $\nu$ in \eqref{eq:l} given in \eqref{eq:mu-nu}.
\subsubsection*{Case $L(0,0) = 0$} This is a simple case. First, taking the limit of \eqref{eq:l} as $x\to0$ gives $\nu=0$. The value of $\mu$ then follows by examining the first term (if $p_{1,1}=0$ and $p_{0,1}\neq 0$) or the second term (if $p_{1,1}=0$ and $p_{0,1}= 0$) in the expansion of \eqref{eq:l} at $x=0$.

Before studying the case $L(0,0)\neq 0$ we need to analyze the behavior of the function $w(x)$ at the boundary point $1$.
\subsubsection*{Singularity of the conformal mapping at the boundary point $1$}
We know from \ref{W-continuous} that $w(1)=\infty$. More precisely, it is proved in \cite{FR2} that there exists $c\neq 0$ such that for $x$ in the neighborhood of $1$,
     \begin{equation}
     \label{behavior_chi}
          w(x)=\frac{c +o(1)}{(1-x)^{\pi/\theta}}.
     \end{equation}
This very important property can be deduced rather easily from the fact that $w(x)$ is a conformal mapping and from the equality $w(x)=w(\overline{x})$ for $x\in X([y_1,1])$ near $1$ (Figure~\ref{fig:ex_curve}).

\subsubsection*{Case $L(0,0)\neq 0$}

A similar reasoning as in Sections \ref{subsec:1} and \ref{subsec:2} gives that
\begin{equation}
\label{eq:ltilde}
     H(0,y) = \widetilde \mu\frac{ \widetilde w(y)+ \widetilde \nu}{L(0,y)},
\end{equation}
where $\widetilde w(y)$ is the analogue of $w(x)$, replacing the domain $\mathscr G_X$ by $\mathscr G_Y$, the domain bounded by the curve $Y([x_1,1])$ and containing $y_1$. Similar properties as \ref{W-boundary}--\ref{W-continuous} hold for $\widetilde w(y)$. Further, we fix $\widetilde w(y)$ so as $\widetilde w(0)=0$ and \eqref{behavior_chi} holds (with exactly the same value of $c$). 

We show that for \eqref{eq:l} and \eqref{eq:ltilde} to hold, we must have $\mu = \widetilde \mu$. To this aim, evaluate the functional equation \eqref{eq:functional_equation} at $(X_0(y),y)$; we obtain
\begin{equation*}
     0= L(X_0(y),0)H(X_0(y),0)+L(0,y)H(0,y)-L(0,0)H(0,0).
\end{equation*}
Using \eqref{eq:l} and \eqref{eq:ltilde}, we may write
\begin{equation}
\label{eq:btm}
     0= \mu(w(X_0(y))+\nu)+\widetilde \mu(\widetilde w(y)+\widetilde \nu)-L(0,0)H(0,0).
\end{equation}
In the neighborhood of $1$, we have $X_0(y) = 1+\exp(i \theta) (y-1)+o(y-1)$, see \cite[Equation (2.8)]{FR2}, so that
\begin{equation*}
     w(X_0(y)) = \frac{c+o(1)}{(1-X_0(y))^{\pi/\theta}} = -\frac{c+o(1)}{(1-y)^{\pi/\theta}}.
\end{equation*}
Multiplying Equation \eqref{eq:btm} by $(1-y)^{\pi/\theta}$ and letting $y\to1$, we obtain that $\mu=\widetilde \mu$.

Now we return to \eqref{eq:l} and \eqref{eq:ltilde}. Evaluating them at $x=0$ and $y=0$, respectively, we obtain that $\nu = \widetilde \nu = p_{1,1}f(1,1)/\mu$. To find $\mu$ in the case $L(0,0)\neq 0$, we can use \eqref{eq:btm}:
\begin{equation*}
     \mu = -\frac{p_{1,1} f(1,1)}{w(X_0(y))+\widetilde w(y)}.
\end{equation*}
The function in the denominator above is constant, and can be evaluated at any value of $y$, for example $0$.

\subsection{A perfect conformal mapping}
\label{mapping}

In this part we study the function $w(x)$ satisfying \ref{W-boundary}--\ref{W-continuous}, we show in particular that it can be expressed in terms of sine and arcsine functions as in \eqref{def_CGF_u}.

Because of \ref{W-conformal} and \ref{W-analytic}, the function $w(x)$ is a conformal mapping, in the sense that it maps conformally the domain $\mathscr{G}_X$ onto some domain (in fact, the complex plane ${\bf C}$ cut along some interval). More precisely, $w(x)$ is a conformal gluing function \cite{FIM,LIT}, as it glues together the upper lower edges of the boundary of $\mathscr{G}_X$ (see \ref{W-boundary}). However, this information (that $w(x)$ is a conformal mapping---or even a conformal gluing) is a priori not sufficient to determine the function $w(x)$ explicitly. Indeed, for a given general domain, it is a very difficult task to find an expression for the associated conformal mapping. In the literature, explicit expressions are available for quite particular domains only (circles, ellipses, rectangles, strips, etc.). In the case we are interested in, the domain $\mathscr{G}_X$ and its boundary $X([y_1,1])$ are explicit: the latter is in general a quartic curve \cite{FIM}; it can be degenerated into circles (for the simple random walk, for instance) or other simpler domains.

Two expressions for the function $w(x)$ have already been found in the literature. First, a study is proposed in \cite[Section 6.5]{FIM}, which ends up in a formula for $w(x)$. The method used there is to transform the curve $X([y_1,1])$ (thanks to a uniformization of the Riemann surface defined by the zeros of the kernel \eqref{eq:def_L}), so as to make the transformed curve simpler (it becomes the union of two circular arcs), and finally to use a standard expression for the conformal mapping of the simpler curve. Second, the function $w(x)$ has also been carefully analyzed in \cite{FR2}, by using different techniques. Namely, the conformal mapping was first computed in the case of random walks with non-zero drift, and it happens that when the drift goes to $0$, the conformal mapping converges to the conformal function in the zero drift case. In particular, the expression \eqref{def_CGF_u} was derived, see \cite[Section 2.2]{FR2}. 

A priori, the function $w(x)$ in \eqref{def_CGF_u} is well defined only for $\vert T\vert=\vert T(x)\vert \leq 1$. For $\vert T\vert\geq 1$, we have $\arcsin (T) = \pi/2\pm i\ln(T+\sqrt{T^2-1})$, and as in \cite[Section 2.2]{FR2} we can write
\begin{equation}
\label{def_CGF_u-alternative}
     w(x) = -\frac{1}{4} \left[\left(T+\sqrt{T^2-1}\right)^{2\pi/\theta}-2+\left(T-\sqrt{T^2-1}\right)^{2\pi/\theta}\right].
\end{equation}

\subsection{Examples}
In this section we consider two examples, the simple random walk (Figure~\ref{SRW}) and a random walk related to the Lie algebra $\mathfrak{sl}_{3}({\bf C})$ (Figure \ref{SecondRW}), for which we compute the harmonic function $f(i_0,j_0)$ by applying Theorem \ref{thm:main-harmonic}.

\subsubsection*{Example 1 (continued)}   

Thanks to Theorem \ref{thm:main-harmonic}, to compute $H(x,0)$ it suffices to calculate $w(x)$, $L(x,0)$, $\mu$ and $\nu$. For the simple random walk, it is easy to compute 
\begin{equation*}
     \left\{\begin{array}{rcl}
     x_4&= & 3+2\sqrt{2},\\
     \delta'(x_4) &=& 4+3\sqrt{2},\\
     \delta''(x_4)&=&11/2+3\sqrt{2},\\
     \delta''(1)&=&-1/2,
     \end{array}\right.
\end{equation*}    
as well as
 \begin{equation*}
     T(x) =\displaystyle  \sqrt{\frac{x-3-2\sqrt{2}}{2(2+\sqrt{2})(x-1)}}.
\end{equation*}
Further, $\pi/\theta=2$, so that with \eqref{def_CGF_u-alternative}, 
\begin{equation*}
\label{def_CGF_u-SRW}
     w(x) = -4T(x)^2(T(x)^2-1) = \frac{(x-3+2\sqrt{2})(x-3-2\sqrt{2})}{2(x-1)^2}=\frac{1}{2}-\frac{2x}{(x-1)^2}.
\end{equation*}
Subtracting by $w(0)$ we find $-2x/(x-1)^2$. Furthermore, we have $L(x,0) = x/4$, so that $L(0,0)=0$, and for this reason $\nu=0$. Finally, with \eqref{eq:mu-nu}, $\mu=f(1,1)/4$, and thus 
     \begin{equation*}
          H(x,0) = \frac{f(1,1)}{(1-x)^2}.
     \end{equation*}
This does match with \eqref{eq:expression_HHH_SRW}.

\subsubsection*{A second example} 

We now turn to a second example, more sophisticated than the simple random walk, whose transition probabilities are represented on Figure \ref{SecondRW}: $p_{1,0}=p_{0,-1}=p_{-1,1}=1/3$. This random walk is well studied in the literature, notably because it can be interpreted as a walk in the Weyl chamber of the Lie algebra $\mathfrak{sl}_{3}({\bf C})$, see \cite{Bi2}. It is proved in \cite{Bi2} that there is a unique harmonic function, namely $f(i_0,j_0)=i_0j_0(i_0+j_0)/2$. Let us retrieve this result by using our techniques (namely, by applying Theorem \ref{thm:main-harmonic}). 

\unitlength=1.1cm
\begin{figure}[t]
        \begin{picture}(6,5.5)
    \thicklines
    \put(1,1){{\vector(1,0){4.5}}}
    \put(1,1){\vector(0,1){4.5}}
    \thinlines
    \put(3,3){\vector(1,0){1}}
    \put(3,3){\vector(-1,1){1}}
    \put(3,3){\vector(0,-1){1}}
    \put(4.05,3.15){$1/3$}
    \put(1.45,4.15){$1/3$}
    \put(3.05,1.7){$1/3$}
    \linethickness{0.1mm}
    \put(1,2){\dottedline{0.1}(0,0)(4.5,0)}
    \put(1,3){\dottedline{0.1}(0,0)(4.5,0)}
    \put(1,4){\dottedline{0.1}(0,0)(4.5,0)}
    \put(1,5){\dottedline{0.1}(0,0)(4.5,0)}
    \put(2,1){\dottedline{0.1}(0,0)(0,4.5)}
    \put(3,1){\dottedline{0.1}(0,0)(0,4.5)}
    \put(4,1){\dottedline{0.1}(0,0)(0,4.5)}
    \put(5,1){\dottedline{0.1}(0,0)(0,4.5)}
\end{picture}
  \vspace{-10mm}
\caption{A second example of random walk in the quarter plane}
\label{SecondRW}
\end{figure}
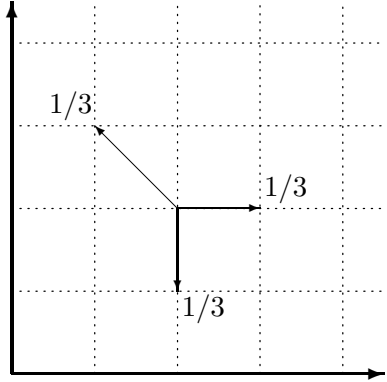

One has $x_4=\infty$. Using the formula \eqref{def_f} we obtain $T(x)=\sqrt{3}/(2\sqrt{1-x})$. Moreover, the quantity $\theta$ is equal to $\pi/3$, see \eqref{exp_theta}, and thus with \eqref{def_CGF_u-alternative} we reach the conclusion that
\begin{align*}
     w(x) &= -\frac{1}{4} \left[\left(T(x)+\sqrt{T(x)^2-1}\right)^6-2+\left(T(x)-\sqrt{T(x)^2-1}\right)^6\right]\\
             &= -(T(x)-1)(T(x)+1)(2T(x)-1)^2(2T(x)+1)^2\\
             &= \frac{(x-1/4)(x+2)^2}{(x-1)^3}.
\end{align*}
Subtracting by $w(0)$ we obtain 
\begin{equation*}
     w(x) = \frac{27}{4}\frac{x^2}{(x-1)^3}.
\end{equation*}
Since $p_{1,1}=p_{0,1}=0$ we have $\nu=0$ in \eqref{eq:mu-nu} and for $\mu$ we use the above formula in \eqref{eq:mu-nu}. After easy computations and using that $L(x,0)=x^2/3$ we obtain
\begin{equation*}
     H(x,0)=\frac{f(1,1)}{(1-x)^3}.
\end{equation*}
Similarly we could prove that 
\begin{equation*}
     H(y,0)=\frac{f(1,1)}{(1-y)^3}.
\end{equation*}
Finally, using the basic functional equation yields
\begin{equation*}
     H(x,y)=\frac{1-xy}{(1-x)^3(1-y)^3} = \textstyle\sum_{i_0,j_0\geq 1}\displaystyle\frac{ i_0j_0(i_0+j_0)}{2}x^{i_0-1}y^{j_0-1},
\end{equation*}
which indeed corresponds to the result in \cite{Bi2}.

\section{Applications}
\label{MB}
\setcounter{equation}{0}

In this section we present three applications of our results, which successively concern Martin boundary, exit time from cones, and lattice path enumeration.

\subsection{Martin boundary}
\label{subsec:MB}

\subsubsection*{Brief account on Martin boundary theory \cite{Dynkin}}

For a (transient) Markov chain with state space $E$, the {Martin compactification} of $E$ is the smallest compactification $\overline{E}$ of $E$ for which the Martin kernels 
\begin{equation*}
     y\mapsto k_{y}^{x}=G_{y}^{x}/G_{y}^{x_0}
\end{equation*}      
extend continuously ($G_{y}^{x}$ is the {Green function}, and we denote by $x_0$ any reference state). The set $\overline{E}\setminus E$ is called the {full Martin boundary}. For  $\alpha\in\overline{E}$, $x\mapsto k_{\alpha}^{x}$ is superharmonic, and 
\begin{equation*}
     \partial_{m}E=\{\alpha \in \overline{E}\setminus E : x\mapsto k_{\alpha}^{x}\ \text{is minimal harmonic}\}
\end{equation*}     
is the {minimal Martin boundary}\footnote{A harmonic function $h$ is minimal if $0\leq g\leq h$ with $g$ harmonic implies $g=c h$ for some constant $c$.}. Then any superharmonic function $h$ can be written as 
\begin{equation*}
     h(x)=\int_{\overline{E}}k_{y}^{x}\mu(\text{d}y)
\end{equation*} 
and any harmonic function $h$ can be written as
\begin{equation*}
     h(x)=\int_{\partial_{m}E}k_{y}^{x}\mu(\text{d}y),
\end{equation*}
where $\mu$ is some finite measure, uniquely characterized in the case just above.

\subsubsection*{Our contribution}
In Section \ref{solving}, provided that the drift of the walk is zero, we have proved the uniqueness of the harmonic function. In other words:

\begin{thm}
\label{thm:MB}
For all random walks satisfying \ref{small_jumps}, \ref{non_degenerate} and \ref{drift}, the Martin boundary is reduced to one point.
\end{thm}

In particular, all results in \cite{Bi2,PW,RaSp4} regarding the uniqueness of the harmonic function for some particular cases of walks in the quarter plane follow from Theorem \ref{thm:MB}. Let us be more specific. Put in the framework of this article, it is proved in \cite{Bi2} that for the model of Figure \ref{SecondRW} there is a unique harmonic function. In \cite{PW} it is shown that there is uniqueness of the harmonic function for any cartesian product of random walks with zero drift (this covers the simple random walk of Figure \ref{SRW}). In \cite{RaSp4} the random walks with jumps as in Figure \ref{Third example} are studied, and the uniqueness of the harmonic function is obtained (via the asymptotic behavior of the Green functions and the theory of Martin boundary). 

\unitlength=1.1cm
\begin{figure}[t]
        \begin{picture}(6,5.5)
    \thicklines
    \put(1,1){{\vector(1,0){4.5}}}
    \put(1,1){\vector(0,1){4.5}}
    \thinlines
    \put(3,3){\vector(1,0){1}}
    \put(3,3){\vector(-1,0){1}}
    \put(3,3){\vector(1,-1){1}}
    \put(3,3){\vector(-1,1){1}}
    \linethickness{0.1mm}
    \put(1,2){\dottedline{0.1}(0,0)(4.5,0)}
    \put(1,3){\dottedline{0.1}(0,0)(4.5,0)}
    \put(1,4){\dottedline{0.1}(0,0)(4.5,0)}
    \put(1,5){\dottedline{0.1}(0,0)(4.5,0)}
    \put(2,1){\dottedline{0.1}(0,0)(0,4.5)}
    \put(3,1){\dottedline{0.1}(0,0)(0,4.5)}
    \put(4,1){\dottedline{0.1}(0,0)(0,4.5)}
    \put(5,1){\dottedline{0.1}(0,0)(0,4.5)}
    \put(4.05,3.2){$p_{1,0}$}
    \put(4.05,2.2){$p_{1,-1}$}
    \put(1.05,3.7){$p_{-1,1}$}
    \put(1.05,2.7){$p_{-1,0}$}
\end{picture}
  \vspace{-10mm}
\caption{A third example of walk in the quarter plane, with transition probabilities $p_{1,0}=p_{-1,0}=\sin(\pi/n)^2/2$ and $p_{1,-1}=p_{-1,1}=\cos(\pi/n)^2/2$, for a fixed value of $n\geq 3$}
\label{Third example}
\end{figure}
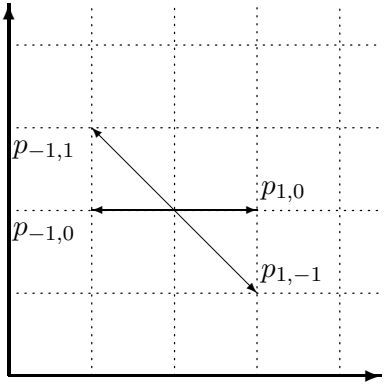

\subsubsection*{A conjecture on the uniqueness of harmonic functions}

In fact, it is our opinion that Theorem \ref{thm:MB} can be extended as follows:
\begin{conj}
For any random walks in the quarter plane with zero drift \ref{drift} and square integrable increments, the Martin boundary is reduced to one point.
\end{conj}

\subsection{Exit time from the cone}
As written in the introduction, it is proved in \cite{DW} that for the zero-mean random walks in the quarter plane, Equation \eqref{eq:asymptotic-DW} holds, where $V(i_0,j_0)$ is some non-explicit harmonic function. Let $f(i_0,j_0)$ be the unique harmonic function found in Section \ref{solving}.

\begin{cor}
\label{cor:exit-time}
There exists $\chi>0$ such that 
\begin{equation}
\label{eq:equality_harmonic_functions}
     V(i_0,j_0)=\chi  f(i_0,j_0),\qquad \forall(i_0,j_0)\in{\bf Z}_+^2.
\end{equation}
\end{cor}

The constant $\chi$ in Corollary \ref{cor:exit-time} can be made explicit, by evaluating the identity \eqref{eq:equality_harmonic_functions} either at $(i_0,j_0)=(1,1)$ or as $i_0+j_0\to\infty$.

The conformal mapping $w(x)$ is related to the first exit time from the cone, as $V(i_0,j_0)$ governs the asymptotic tail distribution of the latter (Corollary \ref{cor:exit-time}). The link seems actually to be  stronger, as the following remark holds: let $F(n)$ and $\theta$ be introduced in \eqref{eq:asymptotic-DW} and  \eqref{exp_theta}, respectively. It is proved in \cite[Theorem 1 and Example 2]{DW} that $F(n) = n^{-\pi/(2\theta)}$, and in \eqref{behavior_chi} that $w(x)=(c+o(1))/(1-x)^{\pi/\theta}$, with exactly the same angle $\theta$.

\subsection{Lattice path enumeration}

A point in the lattice ${\bf Z}_+^2$ has eight neighbors, diagonal neighbors included. Choosing a subset of the set of neighbors and considering it as the set of admissible steps of a walk in the lattice, we obtain one of $2^8$ possible models of walks in ${\bf Z}_+^2$, see Figure \ref{fig:counting_walks} for an example. Bousquet-M\'elou and Mishna \cite{BMM} showed that, after eliminating trivial cases, and also those which can be reduced to the walks in a half plane, there remain $79$ inherently different models. Recently, in \cite[Section 1.5]{DW}, the asymptotics of the number $N_n(x,y)$ of excursions (the number of walks starting at $x$ and ending at $y$ after $n$ steps) was obtained: 
\begin{equation*}
     N_n(x,y) = C V(x) V'(y) \rho^n n^{\alpha}(1+o(1)),\qquad n\to\infty,
\end{equation*}
where the constants $C$, $\rho$ and $\alpha$ are found explicitly, while $V(x)$ and $V'(y)$ are not. In fact, with the notations of Section \ref{explicit}, $V(x)$ is the harmonic function for the random walk $(X,Y)$, and $V'(y)$ is the harmonic function for the reversed process $-(X,Y)$. If the drift of the walk (see \ref{drift}) is zero, then the drift of the reversed walk is zero too, and the results of Section \ref{solving} provide expressions for both $V(x)$ and $V'(y)$. If the drift is non-zero, the harmonic functions have been found in \cite{KuRa}. They could also be obtained from Section~\ref{Extensions} below, where we present some extensions (for instance, to the non-zero drift case) of our approach and results.

\unitlength=1.1cm
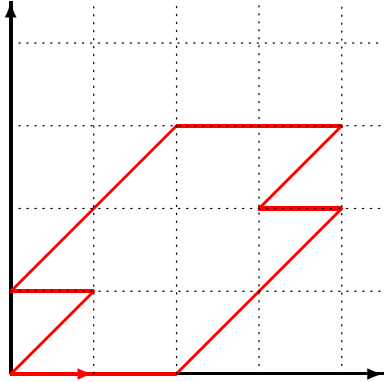
\begin{figure}[t]
        \begin{picture}(6,5.5)
    \thicklines
    \put(1,1){{\vector(1,0){4.5}}}
    \put(1,1){\vector(0,1){4.5}}
    \thinlines
    \Thicklines
    \put(1,1){\textcolor{red}{\vector(1,0){1}}}
    \put(2,1){\textcolor{red}{\line(1,0){1}}}
    \put(3,1){\textcolor{red}{\line(1,1){1}}}
    \put(4,2){\textcolor{red}{\line(1,1){1}}}
    \put(5,3){\textcolor{red}{\line(-1,0){1}}}
    \put(4,3){\textcolor{red}{\line(1,1){1}}}
    \put(5,4){\textcolor{red}{\line(-1,0){1}}}
    \put(4,4){\textcolor{red}{\line(-1,0){1}}}
    \put(3,4){\textcolor{red}{\line(-1,-1){1}}}
    \put(2,3){\textcolor{red}{\line(-1,-1){1}}}
    \put(1,2){\textcolor{red}{\line(1,0){1}}}
    \put(2,2){\textcolor{red}{\line(-1,-1){1}}}
    \linethickness{0.1mm}
    \thinlines
    \put(1,2){\dottedline{0.1}(0,0)(4.5,0)}
    \put(1,3){\dottedline{0.1}(0,0)(4.5,0)}
    \put(1,4){\dottedline{0.1}(0,0)(4.5,0)}
    \put(1,5){\dottedline{0.1}(0,0)(4.5,0)}
    \put(2,1){\dottedline{0.1}(0,0)(0,4.5)}
    \put(3,1){\dottedline{0.1}(0,0)(0,4.5)}
    \put(4,1){\dottedline{0.1}(0,0)(0,4.5)}
    \put(5,1){\dottedline{0.1}(0,0)(0,4.5)}
\end{picture}
  \vspace{-10mm}
\caption{Counting walks in the quarter plane, with the example of the so-called Gessel's step set $\{(1,0),(1,1),(-1,0),(-1,-1)\}$ (see \cite{BMM} for a general introduction to this topic)}
\label{fig:counting_walks}
\end{figure}

\section{Further properties of discrete harmonic functions}
\label{sec:further_properties}
\setcounter{equation}{0}

\subsection{Comparison between continuous and discrete harmonic functions}
In this section we compare the continuous and discrete harmonic functions. We first write down the expression of the unique continuous harmonic function. Then we illustrate that for some particular cases, the continuous and discrete harmonic functions are exactly the same. Finally we state a conjecture on the asymptotic behavior of discrete harmonic functions as $i_0+j_0\to\infty$.

\subsubsection*{Expression of the classical harmonic function}
The standard Laplacian is
\begin{equation*}
     \Delta=\frac{\partial^2}{\partial x^2}+\frac{\partial^2}{\partial y^2}.
\end{equation*}
It corresponds to a process with identity covariance matrix. With the covariance matrix 
\begin{equation*}
     \left(\begin{array}{ll}
     \sigma_{11}&\sigma_{12}\\
     \sigma_{21}&\sigma_{22}
     \end{array}\right)=
     \left(\begin{array}{ll}
     \textstyle\sum_{i,j}\displaystyle i^2p_{i,j}&\textstyle\sum_{i,j}\displaystyle ijp_{i,j}\\
     \textstyle\sum_{i,j}\displaystyle ijp_{i,j}&\textstyle\sum_{i,j}\displaystyle j^2p_{i,j}
     \end{array}\right)
\end{equation*}
(with $\sigma_{12}=\sigma_{21}$) of the random walk, the Laplacian becomes
\begin{equation*}
     L=\sigma_{11}\frac{\partial^2}{\partial x^2}+2\sigma_{12} \frac{\partial^2}{\partial x\partial y}+\sigma_{22} \frac{\partial^2}{\partial y^2}.
\end{equation*}
There is a unique function $u(x,y)$ which is positive harmonic in the interior of the quarter plane ($u>0$ and $Lu=0$) and equal to zero on the boundary \cite{0513885}. To write its expression, we need to introduce (with $\theta$ defined in \eqref{exp_theta})
\begin{equation}
\label{eq:change_variable}
     \phi (x,y)= \left(\frac{1}{\sin(\theta)}\frac{x}{\sqrt{\sigma_{11}}} - \frac{1}{\tan (\theta)} \frac{y}{\sqrt{\sigma_{22}}} , \frac{y}{\sqrt{\sigma_{22}}}\right).
\end{equation}
Note that $\phi({\bf R}_+^2)$ is the cone $\{\rho \exp(it): 0\leq \rho<  \infty, 0\leq t\leq \theta\}$, and that the random process $\phi(X,Y)$ has a covariance matrix equal to the identity, see \cite[Section 5]{AIM}. Define in polar coordinates $(x,y)=(\rho \cos(t),\rho\sin(t))$ the function $v(x,y)=\rho^{\pi/\theta}\sin(t\pi/\theta)$. Then the harmonic function is given by:
\begin{equation}
\label{eq:expression_continuous_harmonic_function}
     u(x,y)=v(\phi(x,y)).
\end{equation}
This function is called the r\'eduite of the cone \cite{0513885}. It is clearly homogeneous\footnote{A function $P(i_0,j_0)$ is homogeneous if there exists $\alpha\in{\bf R}_+$ such that $P(\lambda i_0,\lambda j_0) = \lambda^\alpha P(i_0,j_0)$ for any $i_0$, $j_0$ and $\lambda$.}. Another way to find an expression for $u(x,y)$ is proposed in Appendix \ref{sec:sandro}.



\subsubsection*{On the equality of the continuous and discrete harmonic functions, and a conjecture on the asymptotic behavior of discrete harmonic functions}

As a consequence of the last paragraph, for the discrete harmonic function to be equal to the classical one, it should be homogeneous. This is the case for the simple random walk (Figure \ref{SRW}), for which $f(i_0,j_0)=i_0j_0$, and for the random walk related to $\mathfrak{sl}_{3}({\bf C})$ (Figure \ref{SecondRW}) as well, for which $f(i_0,j_0)=i_0j_0(i_0+j_0)/2$. In that cases the continuous and discrete functions coincide, using \eqref{eq:expression_continuous_harmonic_function}. This remark also holds for the random walk on Figure \ref{Third example} when $n\in\{3,4\}$. Indeed, it is proved in \cite{RaSp4} that the functions 
\begin{align*}
     f(i_{0},j_{0})&=i_{0} j_{0} (i_{0}+2j_{0}),\\
     f(i_{0},j_{0})&=i_{0} j_{0} (i_{0}+2j_{0}) (i_{0}+j_{0}),
\end{align*}
are the harmonic functions when $n=3$ and $n=4$, respectively. On the other hand, it is shown in \cite{RaSp4} that the harmonic function for the random walk on Figure \ref{Third example} is not homogeneous as soon as $n\geq 5$. As an illustration, for $n=6$, the harmonic function is
    \begin{equation*}
          f(i_{0},j_{0})=i_{0} j_{0}(i_{0}+2j_{0})(i_{0}+j_{0})((i_{0}+2j_{0}/3)(i_{0}+4j_{0}/3)+10/9).
     \end{equation*}
However, for the above example, and in fact in many examples, one has (with $u$ as in \eqref{eq:expression_continuous_harmonic_function})
\begin{equation}
\label{eq:asymptotic_equality}
     f(i_0,j_0)=u(i_0,j_0)(1+o(1)),\qquad i_0+j_0\to\infty.
\end{equation}
We have the following conjecture:

\begin{conj}
\label{conj:limit}
Equation \eqref{eq:asymptotic_equality} holds for all models of random walks in the quarter plane with square integrable increments.
\end{conj}
The equality \eqref{eq:asymptotic_equality} is obvious when the discrete harmonic function is homogeneous. Further, it is proved for all random walks of Figure \ref{Third example} in \cite{RaSp4}. We can prove the conjecture for all small step random walks in the quarter plane for $i_0=1$ (or $j_0=1$), using our main result (Theorem \ref{eq:l}), the behavior \eqref{behavior_chi} of the conformal mapping near $1$, and classical Tauberian theorems. Moreover, Corollary \ref{cor:exit-time} together with \cite[Lemma 13]{DW} imply that Conjecture \ref{conj:limit} holds for all small step random walks in the quarter plane.

\subsection{Group of the walk and polynomial harmonic functions}

In this section we relate some properties of harmonic functions to the group of the walk, a notion introduced by Malyshev \cite{MAL}. First we state a conjecture (Conjecture \ref{conj:equality_group}) on the homogeneity of harmonic functions in terms of the group of the walk. Then we demonstrate a result (Proposition~\ref{prop:nature-H(x,y)}) on the algebraic nature of the generating function $H(x,y)$ in terms of the group. First of all we have to properly introduce the notion of group.

\subsubsection*{Definition of the group (first version)}
Define $Q(x,y)=x^2y^2L(1/x,1/y)$, with $L$ as in \eqref{eq:def_L}, and introduce 
     \begin{equation*}
     \label{def_generators_group}
          \xi(x,y)= \left(x,\frac{1}{y}\frac{\sum_{-1\leq i \leq 1}p_{i,-1}x^{i}}
          {\sum_{-1\leq i \leq 1}p_{i,+1}x^{i}}\right),
          \qquad \eta(x,y)=\left(\frac{1}{x}\frac{\sum_{-1\leq j\leq 1}p_{-1,j}y^{j}}
          {\sum_{-1\leq j\leq 1}p_{+1,j}y^{j}},y\right).
     \end{equation*}
The functions $\xi$ and $\eta$ are transformations of ${\bf C}^2$. It is quite easy to see that if $Q(x,y)=0$, then $Q(\xi(x,y))=Q(\eta(x,y)) = 0$. Though this is not the original definition of Malyshev, we define the group of the walk (denoted by $W_1$) as the group $\langle \xi,\eta\rangle$ of transformations of ${\bf C}^2$. The transformations $\xi$ and $\eta$ are clearly involutions, so that the group $\langle \xi,\eta\rangle$ is a dihedral group. It can be finite or infinite, depending on the order of the element $\xi\circ\eta$ (or $\eta\circ\xi$). We now give some examples. For the simple random walk we have
\begin{equation*}
     \xi(x,y)= \left(x,\frac{1}{y}\right),\qquad \eta(x,y)=\left(\frac{1}{x},y\right),
\end{equation*}
in such a way that the group $\langle \xi,\eta\rangle$ is equal to $\{1,\xi,\eta,\xi\circ\eta=\eta\circ\xi\}$, it is of order $4$. A simple verification shows that for the random walks on Figure \ref{SecondRW} and on Figure \ref{Third example} for $n\in\{3,4\}$, the group is finite, or order $6$, $6$ and $8$, respectively. Our conjecture (which is supported by the above examples) is the following:
\begin{conj}
\label{conj:equality_group}
If the discrete and continuous harmonic functions coincide then the group $W_1$ is finite.
\end{conj}

\subsubsection*{Nature of the generating function $H(x,y)$} Let $\theta$ be defined in \eqref{exp_theta}. We have the following link between the nature (rational, algebraic, non-algebraic) of $H(x,y)$ and $\theta$:

\begin{prop}
\label{prop:nature-H(x,y)}
The generating function $H(x,y)$ is rational {\rm(}resp.\ algebraic non-rational, non-algebraic{\rm)} if and only if $\pi/\theta$ belongs to ${\bf Z}$ {\rm(}resp.\ ${\bf Q}$, ${\bf R}\setminus {\bf Q}${\rm)}.
\end{prop}

\begin{proof}
From the functional equation \eqref{eq:functional_equation} and by symmetry, it suffices to prove the above statement for $H(x,0)$. In fact, thanks to Theorem \ref{thm:main-harmonic}, it is enough to prove it for $w(x)$. If $\pi/\theta$ belongs to ${\bf Z}$ or ${\bf Q}$, Proposition \ref{prop:nature-H(x,y)} easily follows from \eqref{def_CGF_u-alternative}. If $\pi/\theta\in{\bf R}\setminus {\bf Q}$, then it is a corollary of \eqref{behavior_chi} that $w(x)$ cannot be an algebraic function.
\end{proof}

A direct consequence of Proposition \ref{prop:nature-H(x,y)} is the following:
\begin{cor}
The function $f(i_0,j_0)$ is polynomial in $i_0,j_0$ if and only if $\pi/\theta\in {\bf Z}$.
\end{cor}

However, it is not simple to check whether $\pi/\theta$ is rational or not; this is why we propose the following alternative way to find the (ir)rationality of $\pi/\theta$, in terms of the group of the walk.

\subsubsection*{Definition of the group (second version)}
We have already remarked that if $Q(x,y)=0$, then $Q(\xi(x,y))=Q(\eta(x,y)) = 0$. In particular, $\xi$ and $\eta$ in \eqref{def_generators_group} can be interpreted not only as transformations of ${\bf C}^2$, but also as automorphisms of the surface
\begin{equation*}
     \{(x,y)\in ({\bf C}\cup \{\infty\})^2: Q(x,y)=0\}.
\end{equation*}
Our second definition of group of the walk (the original one, see \cite{MAL} and \cite[Section 2.4]{FIM}) is then precisely the automorphism group $W_2=\langle \xi,\eta\rangle$ generated by $\xi$ and $\eta$. The difference between $W_1$ and $W_2$ is not only of a formal character, and in general the notions of order and finiteness do not coincide in the respective approaches. However, it is remarked in \cite[Section 2.1]{FR1} that the order of $W_1$ is larger than (or equal to) the order of $W_2$.

The main point is that under \ref{small_jumps}, \ref{non_degenerate} and \ref{drift}, the group $W_2$ is finite if and only if $\pi/\theta\in{\bf Q}$, see \cite{FR2}. In other words, we can reformulate Proposition \ref{prop:nature-H(x,y)} by replacing the hypothesis that $\pi/\theta\in{\bf Q}$ (resp.\ $\pi/\theta\in{\bf R}\setminus {\bf Q}$) by the assumption that the group $W_2$ is finite (resp.\ infinite). The advantage of this reformulation is that in many examples \cite{BMM,FIM,KuRa,RaSp4}, the order of $W_2$ is already computed.

\subsubsection*{Example 1 (continued)}  For the simple random walk, we saw that $\pi/\theta =2$; therefore, the group $W_2$ must be finite. In fact $W_1$ has order $4$, so that $W_2$ is finite, of order $4$. Note that the order of the group $W$ is twice $\pi/\theta$, which is not surprising: in general, the order of $W$ is $2\inf\{p\in{\bf Z}_+\setminus \{0\}: p \theta/\pi \in {\bf Z}\}$.

\section{Extensions}
\label{Extensions}
\setcounter{equation}{0}

Throughout this paper, we mostly assumed that the random walks had zero drift \ref{drift}. However, our approach can also be applied to obtain the Martin boundary in the non-zero drift case (Section \ref{subsec:SRWNZD}), and to compute $t$-harmonic functions (Section \ref{sec:thf}); we also say a word on a possible generalization to random walks with more general jumps (Section \ref{sec:mjj}).

\subsection{Random walks with non-zero drift}
\label{subsec:SRWNZD}

In this section, we consider random walks satisfying \ref{small_jumps} and \ref{non_degenerate} (but not \ref{drift}). We first recall the existing results in the literature on Martin boundary \cite{IRL,KuRa}. Then we show how the approach in this paper allows us to retrieve these results. Finally, we present an interesting phenomenon on the convergence of the harmonic functions of the non-zero drift case to the unique harmonic function in the zero drift case. In this section we focus on ideas, and therefore just sketch the proofs. 

\subsubsection*{Existing results}
For random walks with non-zero drift killed at the boundary of the quarter plane, the minimal Martin boundary (see Section \ref{subsec:MB} for a brief account on Martin boundary theory) is homeomorphic to $[0,\pi/2]$, see \cite{IRL,KuRa}. In concrete terms, the set of all minimal harmonic functions can be parametrized (in a one-to-one way) by $\gamma\in[0,\pi/2]$. Further, an expression for the harmonic functions $f_\gamma$ is given in \cite{KuRa}, in terms of solutions to certain boundary value problems. While this expression is typically complex, it may admit nice simplifications (in fact, when the group $W_2$ introduced in Section \ref{sec:further_properties} is finite, see \cite{KuRa}).


%
%

\subsubsection*{Retrieving the existing results} In this paragraph, we still consider non-zero drift random walks, and we explain how our methods provide the Martin boundary, and even expressions for harmonic functions. In the zero drift case, the point $1$ plays a crucial role, in particular for the boundary value problem of Section \ref{ASFE}, see \ref{class_function}. This is a consequence of the radius of convergence (equal to $1$) of $H(x,0)$---or equivalently of the non-exponential growth of $f(i_0,1)$ as $i_0\to\infty$. For non-zero drift walks, the growth of $f(i_0,1)$ is now exponential. In fact, the set of admissible $p$ could be characterized, say $p\in I$, where $I\subset (1,\infty)$ is some segment.

Instead of one single Riemann boundary value problem (Lemma \ref{lem:BPV}), we would have infinitely many boundary value problems, parametrized by $p$, as follows:
     \begin{enumerate}[label={\rm($\widehat{\text{\roman{*}}})$},ref={($\widehat{\text{\roman{*}}}$)}]
          \item \label{class_function-2} $H(x,0)$ is analytic in $\mathscr{G}_X$, except at $p\in I$, where it has a pole\footnote{The pole is always simple, except if $p$ lies on the boundary of the domain $\mathscr{G}_X$, in which case the pole is double.\label{footnote:pole}}; 
          \item \label{class_function_2-2} $H(x,0)$ is continuous on $\overline{\mathscr{G}_X}\setminus \{p\}$;
          \item \label{boundary_condition-2} For all $x$ on the boundary of $\mathscr{G}_X$,
          $
               L(x,0)H(x,0)-L(\overline{x},0)H(\overline{x},0)=0.
          $
     \end{enumerate}
The resolution of \ref{class_function-2}, \ref{class_function_2-2} and \ref{boundary_condition-2} (that we shall not tackle here) would use an analogue of the conformal mapping $w$, namely, a function such that:
\begin{enumerate}[label=($\widehat{{\rm W}}$\arabic{*}),ref={\rm ($\widehat{{\rm W}}$\arabic{*})}]
     \item\label{W-boundary-2}For all $x$ on the boundary of $\mathscr{G}_X$, $w(x)=w(\overline{x})$;
     \item\label{W-conformal-2}$w$ is injective in $\mathscr{G}_X$;
     \item\label{W-analytic-2} $w$ is analytic on $\mathscr{G}_X$;
     \item\label{W-continuous-2} $w$ is continuous on $\overline{\mathscr{G}_X}\setminus \{p\}$ and $w( p )=\infty$.
\end{enumerate}
At this point, we wish to emphasize the following main difference between the zero drift case and the non-zero drift case: the expression of $w(x)$ becomes more elaborate (more precisely, in place of sine and arcsine functions, it now involves two $\wp$-Weierstrass functions with different periods\footnote{This is due to the fact that a certain Riemann surface---namely, $\{(x,y)\in({\bf C}\cup\{\infty\})^2: L(x,y)=0\}$---has genus $0$ in the zero drift case and genus $1$ in the non-zero drift case.}). Despite this intrinsic complexity, expressions for $w(x)$ are available in the literature: see \cite[Section 3.1]{KuRa} for a particular value of $p$, and \cite[Section 2]{Ra} for simplifications of $w(x)$ in some particular cases. Once $w(x)$ is known, Equation \eqref{eq:l} holds, and an explicit expression for $H(x,0)$ follows. 

In some sense, this approach is more efficient than the one used in \cite{KuRa}, since we give here an expression for the harmonic functions directly in terms of the function $w(x)$, and not in terms of solutions to some other boundary value problems.


\subsubsection*{Convergence of the harmonic functions of the non-zero drift case to the harmonic function in the zero drift case}
Consider a random walk without drift, and with jumps $\{p_{i,j}\}_{-1\leq i,j\leq 1}$. Denote by $f$ its unique harmonic function.  Let us introduce any family of random walks with non-zero drift and jumps $\{q_{i,j}^n\}_{-1\leq i,j\leq 1}$, such that for all $i$ and $j$, $\lim_{n\to\infty}q_{i,j}^n = p_{i,j}$. For any $n$, there exist infinitely many harmonic functions $f_\gamma^n$, parametrized by $\gamma\in[0,\pi/2]$. We normalize all these functions by requiring that $f(1,1)=f_\gamma^n(1,1)=1$. The main result in this paragraph is the following convergence result: for any $i_0,j_0$ and any $\gamma$, we have 
\begin{equation}
\label{eq:convergence-harmonic-functions}
     \lim_{n\to\infty}f_\gamma^n(i_0,j_0) = f(i_0,j_0).
\end{equation}
We now mention two open problems concerning \eqref{eq:convergence-harmonic-functions}:
\begin{itemize}
     \item The proof that we could write for \eqref{eq:convergence-harmonic-functions} is purely analytical. It is an open problem to decide whether this can be obtained by using only general analysis of harmonic functions.
     \item The convergence \eqref{eq:convergence-harmonic-functions} to the same harmonic function $f(i_0,j_0)$ can be proved independently of knowing that in the zero drift case, there exists a unique harmonic function. It is an open problem to prove that the convergence \eqref{eq:convergence-harmonic-functions} implies the uniqueness of the harmonic function in the zero drift case.
\end{itemize}


\subsection{$t$-harmonic functions}
\label{sec:thf}
In this work, we have determined the harmonic functions, that is, the $t$-harmonic functions (satisfying $P f = t f$) for $t=1$. It is also interesting to consider $t$-harmonic functions other values of $t$ (many motivations are given in \cite[Chapter~2]{Woess}). It turns out that our approach can be extended to find the latter: indeed, the key functional equation \eqref{eq:functional_equation} still holds, if now
     \begin{equation*}
          L(x,y)=L(x,y;t)= x y[ \textstyle\sum_{-1\leq i,j\leq 1}p_{i,j }x^{-i} y^{-j}  -t].
     \end{equation*}
Accordingly, the methods presented in this article can be used to find $t$-harmonic functions, with, nevertheless, the following additional difficulty concerning the conformal mapping $w(x)$: mutatis mutandis, as for the non-zero drift case (see Section \ref{subsec:SRWNZD}), introducing the variable $t$ would have the effect of involving $\wp$-Weierstrass functions instead of sine and arcsin functions in the expression of the crucial conformal mapping $w(x)$.

\subsection{More general jumps}
\label{sec:mjj}
In this part we consider random walks in the quarter plane, that do not satisfy the small step hypothesis \ref{small_jumps}. Let us only assume that it has bounded jumps, which means that there exist $I^\pm$ and $J^\pm$ such that $p_{i,j} = 0$ if $i>I^+$, $i<I-$, $j>J^+$ or $j<J^-$. We are interested in harmonic functions satisfying \ref{property_harmonicity}, \ref{property_zero}\footnote{In fact, \ref{property_zero} has to be replaced by $f(i_0,j_0)=0$ if $i_0\leq 0$ or $j_0\leq 0$.} and \ref{property_positive}. Let us define the sectional generating functions
\begin{equation*}
     F^{j_0}(x) = \textstyle\sum_{i_0\geq 1} f(i_0,j_0) x^{i_0-1},
     \qquad
     G^{i_0}(y) = \textstyle\sum_{j_0\geq 1} f(i_0,j_0) y^{j_0-1}.
\end{equation*}
The following functional equation then holds:
\begin{align}
\label{eq:functional_equation_general}
     L(x,y)H(x,y) &=\textstyle \sum_{1\leq j_0\leq J^+} F^{j_0}(x)y^{j_0-1} \sum_{\substack{I^-\leq i\leq I^+\\j_0\leq j \leq J^+}}p_{i,j}x^{-(i-1)}y^{-(j-1)}\\&+\textstyle\sum_{1\leq i_0\leq I^+} G^{i_0}(y)x^{i_0-1} \sum_{\substack{i_0\leq i \leq I^+\\ J^-\leq j\leq J^+}}p_{i,j}x^{-(i-1)}y^{-(j-1)}\nonumber\\&\textstyle-\sum_{\substack{1\leq i_0\leq I^+\\1\leq j_0\leq J^+}}f(i_0,j_0)x^{i_0-1}y^{j_0-1} \sum_{\substack{i_0\leq i \leq I^+\\ j_0\leq j\leq J^+}}p_{i,j}x^{-(i-1)}y^{-(j-1)}.\nonumber
\end{align}
This equation is far more complicated than \eqref{eq:functional_equation}, because instead of $2$ unknowns (namely, $H(x,0)$ and $H(0,y)$) in the right part of \eqref{eq:functional_equation}, there are now $I^++J^+$ unknowns (namely, $F^{j_0}(x)$ for $1\leq j_0\leq J^+$ and $G^{i_0}(y)$ for $1\leq i_0\leq I^+$) in the RHS of \eqref{eq:functional_equation_general}. To our knowledge, it is an open problem to solve \eqref{eq:functional_equation_general} in general.

However, if we restrict ourselves to $I^+ = J^+=1$ (with $I^-$ and $J^-$ arbitrarily large), then \eqref{eq:functional_equation_general} is nothing else but \eqref{eq:functional_equation}, as $F^1(x) = H(x,0)$ and $G^1(y) = H(0,y)$. In particular, the approach developed in the present paper can be applied for any such random walk, provided that a conformal mapping $w(x)$ can be found: if $I^-$ or $J^-$ are different from $-1$, then the kernel $L(x,y)$ becomes more complicated, and it is an open problem to construct a suitable conformal mapping $w(x)$ with properties \ref{W-boundary}--\ref{W-continuous}.

\appendix

\section{Finding harmonic functions for Brownian motion via functional equations (By Sandro Franceschi)}
\label{sec:sandro}
\setcounter{equation}{0}

In this appendix we generalize to the continuous case the approach presented in this article. We are interested in finding harmonic functions associated to the Brownian motion in the quarter plane ${\bf R}_+^2$ with covariance matrix 
$$\Sigma =  \begin{pmatrix}
\sigma_{11} & \sigma_{12} \\
\sigma_{12} & \sigma_{22}
\end{pmatrix}.$$ 
The latter are the functions which are positive within the quarter plane, which vanish on the boundary axes, and which cancel the generator
$$\mathcal{G} f = \frac{1}{2} 
\left(
\sigma_{11} \frac{\partial^2 f}{\partial x^2} + 
2 \sigma_{12} \frac{\partial^2 f}{\partial x \partial y}
+ \sigma_{22} \frac{\partial^2 f}{\partial y^2}
 \right). $$

Our first task is to find a functional equation. The Laplace transforms will play the role of the generating functions. For $\vartheta = (\vartheta_1 , \vartheta_2) $ and $f$ a function with a non-exponential growth, we shall note
$$L (f)(\vartheta) = \int_0^{\infty} \int_0^{\infty}  f(x,y) e^{-(\vartheta_1 x + \vartheta_2 y)} \mathrm{d}x  \mathrm{d}y, \qquad \forall \Re \vartheta_i >0.$$
The functional equation takes the form
\begin{equation}
\label{eq:fec}
\gamma (\vartheta) L(f)(\vartheta) = \frac{1}{2} \left[
 \sigma_{11} L \left(\frac{\partial f}{\partial x} (0,\cdot)\right) (\vartheta_2 )
 + 
  \sigma_{22}  L \left( \frac{\partial f}{\partial y} (\cdot,0)\right) (\vartheta_1 )
  \right],
\end{equation}
where we have noted
\begin{equation}
\label{eq:kernel_continuous}
\gamma (\vartheta) = \frac{1}{2}\textstyle \sum_{i,j=1}^2 \displaystyle\sigma_{ij} \vartheta_1^2
=  \frac{1}{2} ( \sigma_{11} \vartheta_1^2
+ 2 \sigma_{12} \vartheta_1 \vartheta_2
+ \sigma_{22} \vartheta_2^2 ).
\end{equation}
The above quantity is called the kernel of the functional equation.
\begin{proof}[Proof of Equation \eqref{eq:fec}]
The proof is based on a simple integral calculus on the Laplace transform. If $f$ is not  growing too quickly, we obtain the three following equalities by doing an integration by part and remembering that $f=0$ on the boundary of the quarter plane:
\begin{align}
L(f)
&= \label{1}
\frac{1}{\vartheta_1^2} 
\left(
\int_0^\infty  \frac{\partial f}{\partial x}(0,y) e^{ -\vartheta_2 y} \mathrm{d}y  
+ \int_0^{\infty} \int_0^{\infty}  \frac{\partial^2 f}{\partial x^2}(x,y) e^{-(\vartheta_1 x + \vartheta_2 y)} \mathrm{d}x  \mathrm{d}y
\right)
\\
&= \label{2}
\frac{1}{\vartheta_2^2} 
\left(
\int_0^\infty  \frac{\partial f}{\partial y}(x,0) e^{ -\vartheta_1 x} \mathrm{d}x 
+ \int_0^{\infty} \int_0^{\infty}  \frac{\partial^2 f}{\partial y^2}(x,y) e^{-(\vartheta_1 x + \vartheta_2 y)} \mathrm{d}x  \mathrm{d}y
\right)
\\ &= \label{3}
 \frac{1}{\vartheta_1 \vartheta_2} 
\int_0^{\infty} \int_0^{\infty}  \frac{\partial^2 f}{\partial x \partial y}(x,y) e^{-(\vartheta_1 x + \vartheta_2 y)} \mathrm{d}x  \mathrm{d}y.
\end{align}
It is then sufficient to make the following linear combination of the above identities
$$
\frac{1}{2} [ \sigma_{11} \vartheta_1^2 \ (\text{\ref{1}})
 +  \sigma_{22} \vartheta_2^2 \  (\text{\ref{2}})
 + 2 \sigma_{12} \vartheta_1 \vartheta_2   (\text{\ref{3}}) ],
$$
in such a way that $\gamma (\vartheta)$ appears on one side and $\mathcal{G}f$ on the other side. We obtain 
\begin{multline*}
\gamma (\vartheta) L(f)(\vartheta) =
   \int_0^{\infty} \int_0^{\infty}  \mathcal{G} f(x,y)  e^{-(\vartheta_1 x + \vartheta_2 y)} \mathrm{d}x  \mathrm{d}y+\\ \frac{1}{2} \left[ \sigma_{11} \int_0^\infty \frac{\partial f}{\partial x} (0,y) e^{-\vartheta_2 y} \mathrm{d} y 
 +
 \sigma_{22} \int_0^\infty \frac{\partial f}{\partial y} (x,0) e^{-\vartheta_1 x} \mathrm{d} x \right]
.
\end{multline*}
We conclude by using the equality $\mathcal{G}f=0$.
\end{proof}

Our second task is to analyze the basic properties of the kernel \eqref{eq:kernel_continuous}. By studying this polynomial equation of degree two and noticing that $\det{\Sigma} \geqslant 0$ we easily find $\Theta_1(\vartheta_2)$ and $\Theta_2(\vartheta_1)$, the linear functions defined by $\gamma (\Theta_1(\vartheta_2), \vartheta_2))
=\gamma (\vartheta_1, \Theta_2(\vartheta_1))
= 0
$:
$$
\Theta_1(\vartheta_2) = \frac{-\sigma_{12} \pm i \sqrt{\det{\Sigma}}}{\sigma_{11}} \vartheta_2,
\qquad
\Theta_2(\vartheta_1) = \frac{-\sigma_{12} \pm i \sqrt{\det{\Sigma}}}{\sigma_{22}} \vartheta_1.
$$
We note $k= (-\sigma_{12} + i \sqrt{\det{\Sigma}})/{\sigma_{11}} = \sqrt{\frac{\sigma_{22}}{\sigma_{11}}} e^{i \theta}
$ and $l=(-\sigma_{12} + i \sqrt{\det{\Sigma}})/{\sigma_{22}} =
\sqrt{\frac{\sigma_{11}}{\sigma_{22}}} e^{i \theta}
$, where $\theta$ is introduced in Equation \eqref{exp_theta}. The angle $\theta$ is such that
$$ \tan (\theta)= - \frac{\sqrt{\det{\Sigma}}}{\sigma_{12}},
\qquad \cos (\theta)= - \frac{\sigma_{12}}{\sqrt{\sigma_{11}\sigma_{22}}},
\qquad
\sin (\theta)=\sqrt{1 - \frac{\sigma_{12}^2}{\sigma_{11} \sigma_{22}}}
.$$ 
With these notations one can rewrite the kernel as follows:
$$\gamma (\vartheta)= (\vartheta_1-k \vartheta_2)(\vartheta_1-  \overline{k} \vartheta_2)= (\vartheta_2-l \vartheta_1)(\vartheta_2-  \overline{l} \vartheta_1),\qquad \forall \vartheta_1,\vartheta_2\in{\bf C}.
$$

The third point consists in finding, starting from the functional equation \eqref{eq:fec}, explicit expressions for $L(\frac{\partial f}{\partial x} (0,\cdot)) (\vartheta_2 )$ and $L( \frac{\partial f}{\partial y} (\cdot,0)) (\vartheta_1 )$, that we will note $L_2 (\vartheta_2)$ and $L_1 (\vartheta_1)$, respectively. Using then the functional equation and the inverse Laplace transform we will find $f$. The scheme of the demonstration happens to be the same as in the discrete case: we are going to show that $L_1 (\vartheta_1)$ satisfies a boundary value problem on a certain domain $\mathscr{G}_{\Theta_1}$ (which is the analogue of the domain $\mathscr{G}_X$ in \eqref{Notations}). This will eventually lead to 
\begin{equation}
\label{eq:expression_L_1}
L_1(\vartheta_1)= \frac{\mu_1}{ \vartheta_1^{{\pi}/{\theta}}},
\end{equation}
where $\mu_1$ is some non-zero constant.

\unitlength=0.6cm
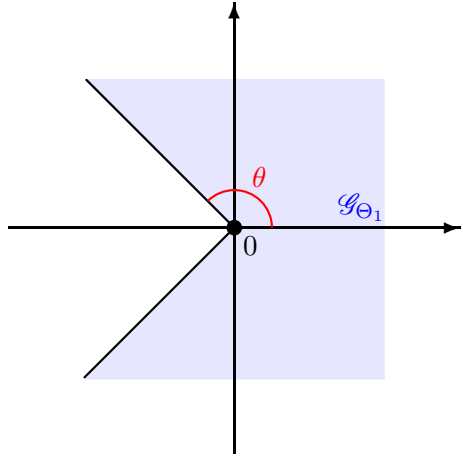
\begin{figure}[t]

\vspace{10mm}
     
         \begin{tikzpicture}
           \linethickness{2mm}
   \draw[draw=none,fill=blue!10] (-2,2) -- (2,2) -- (2,-2) -- (-2,-2) -- (0,0) -- cycle;
     \end{tikzpicture}
     
     \vspace{-25mm}

    \begin{picture}(0,-15)
    \thicklines
    {\put(0,0){\textcolor{black}{\circle*{0.3}}}}
    \put(-5,0){\vector(1,0){10}}
    \put(0,-5){\vector(0,1){10}}
    \put(0,0){\line(-1,1){3.28}}
    \put(0,0){\line(-1,-1){3.32}}
    {\psarc[linecolor=red](0,0){0.5}{0}{135}}
    \put(0.4,0.9){\textcolor{red}{$\theta$}}
    \put(2.3,0.3){\textcolor{blue}{$\mathscr{G}_{\Theta_1}$}}
    \put(0.2,-0.6){$0$}
     \end{picture}

\vspace{27.5mm}

\caption{The domain $\mathscr{G}_{\Theta_1}$ is a cone of opening angle $2\theta$}
\label{fig:ex_domain}
\end{figure}

\begin{proof}[Proof of Equation \eqref{eq:expression_L_1}]
Let us first introduce $\mathscr{G}_{\Theta_1}$ as the open set bounded by the curve $\Theta_1([0, \infty])$,
see Figure \ref{fig:ex_domain} for a graphical representation.
Thanks to an analytic continuation, we shall extend the domain of definition of $L_1$ from $\{\vartheta_1\in{\bf C} :\Re \vartheta_1 > 0 \}$ to $\mathscr{G}_{\Theta_1}$. On the set
$$\{ (\vartheta_1,\vartheta_2)\in {\bf C}^2 : \Re \vartheta_1 >0 , \Re \vartheta_2>0 , \gamma(\vartheta)=0 \}$$ 
we have $\sigma_{11} L_2(\vartheta_2) + \sigma_{22} L_1(\vartheta_1)=0$. This means that  $\sigma_{11} L_2( \vartheta_2) + \sigma_{22} L_1({k} \vartheta_2)=0$ and $\sigma_{11} L_2( \vartheta_2) + \sigma_{22} L_1({\overline k} \vartheta_2)=0$ for all $\vartheta_2$ such as $\Re \vartheta_2 >0$. Thanks to these relations, $L_1$ is now well defined and continuous on $\overline{\mathscr{G}_{\Theta_1}}\setminus \{0\}$ and analytic on $\mathscr{G}_{\Theta_1}$ (this is the analogue of Lemma \ref{lem:BPV} conditions {\rm\ref{class_function}} and {\rm\ref{class_function_2}}). Furthermore we have $L_1(k t)=L_1({\overline{k}} t)$ for all $t\in(0,\infty)$, which means that $L_1(\vartheta_1)=L_1({\overline{\vartheta_1}})$ for all $\vartheta_1$ in the boundary of $\mathscr{G}_{\Theta_1}$ except at $0$ (which is the analogue of Lemma \ref{lem:BPV} {\rm\ref{boundary_condition}}). Therefore, we have shown that $L_1 (\vartheta_1)$ satisfies the same boundary value problem as the one stated in Lemma \ref{lem:BPV}.

To solve this boundary value problem we introduce $w(\vartheta_1)=1/{\vartheta_1^{{\pi}/{\theta}}}$. Obviously $w(\vartheta_1)$ is a conformal mapping from $\mathscr{G}_{\Theta_1}$ to ${\bf C}\setminus {\bf R}_{-}$, which besides satisfies $w(k t)=w({\overline{k}} t)$ for all $t\in(0,\infty)$. To avoid any problem at $0$, rather than $L_1(\vartheta_1)$ we consider $\mathcal K(\vartheta_1)$, defined by
$$
\mathcal K(\vartheta_1) = \frac{1}{L_1(\vartheta_1) - z_0}
$$
for any given non-real $z_0$.
The function $\mathcal{K}(\vartheta_1)$ has a finite number of poles $x_1,\ldots ,x_q$ on ${\mathscr{G}_{\Theta_1}}$, that we can cancel thanks to $w$ and some polynomials $Q_\ell$, according to the same procedure as in Section \ref{subsec:1}. More specifically, it is possible to construct 
$$
D(\vartheta_1)= \textstyle\sum \displaystyle_{\ell=1}^q Q_\ell \left( \frac{1}{w(\vartheta_1)-w(x_\ell)} \right)
$$
such that $\mathcal{K}(\vartheta_1) - D(\vartheta_1)$ is well defined and continuous on $\overline{\mathscr{G}_{\Theta_1}}$ (condition Lemma \ref{lem_constant} {\rm\ref{class_function_2_tilde}}) and is analytic on ${\mathscr{G}_{\Theta_1}}$ (condition Lemma \ref{lem:BPV} {\rm\ref{class_function}}). Furthermore for all $\vartheta_1$ on the boundary of $\mathscr{G}_{\Theta_1}$,
          $
              \mathcal{K}(\vartheta_1) - D(\vartheta_1)=\mathcal{K}(\overline{\vartheta_1}) - D(\overline{\vartheta_1})
          $, i.e.,
           $
              \mathcal{K}(k t) - D(k t)=\mathcal{K}(\overline{k} t) - D(\overline{k} t)
          $ for all $t\in[0,\infty)$ (condition Lemma \ref{lem:BPV} {\rm\ref{boundary_condition}}). 
With the key Lemma \ref{lem_constant}, we deduce that $\mathcal K(\vartheta_1)-D(\vartheta_1)$ is a constant function, which turns out to be $0$.

We readily deduce that $L_1(\vartheta_1)= z_0 + {1}/{S(w(\vartheta_1))}$, where $$S(X)=\textstyle\sum_{\ell=1}^q\displaystyle Q_\ell \left( \frac{1}{X-w(x_\ell)} \right).$$ As for the discrete case it is possible to show that ${1}/{S(X)}$ must be a polynomial, and we finally obtain the formula \eqref{eq:expression_L_1}.
\end{proof}

Similarly, one has the expression (where $\mu_2$ is some non-zero constant)
\begin{equation}
\label{eq:expression_L_2}
L_2(\vartheta_2)= \frac{\mu_2}{ \vartheta_2^{{\pi}/{\theta}}}.
\end{equation}
Further, the constants $\mu_1$ and $\mu_2$ in \eqref{eq:expression_L_1} and \eqref{eq:expression_L_2} are related together: evaluating the functional equation at any point such that $\gamma(\vartheta)=0$, we obtain
\begin{equation*}
     \mu_2=\mu_1 \left(\frac{\sigma_{22}}{\sigma_{11}}\right)^{1-\pi/(2\theta)}.
\end{equation*}

For example, if we take $\Sigma $ equal to the identity matrix and thus $\theta = {\pi}/{2}$, the functional equation \eqref{eq:fec} together with \eqref{eq:expression_L_1} and \eqref{eq:expression_L_2} provide 
$$L(f)(\vartheta)= \frac{1}{\vartheta_1^2 + \vartheta_2^2} \left( \frac{1}{\vartheta_1^2} +\frac{1}{\vartheta_2^2} \right) = \frac{1}{\vartheta_1^2 \vartheta_2^2},$$
which is the Laplace transform of the function $f(x,y)=xy$.

In the general case we find
$$L(f)(\vartheta)= 
\frac{\sigma_{11} {\mu_1}/{ \vartheta_1^{{\pi}/{\theta}}}
+ \sigma_{22} \mu_2/{ \vartheta_2^{{\pi}/{\theta}}}}{\sigma_{11} \vartheta_1^2
+ 2 \sigma_{12} \vartheta_1 \vartheta_2
+ \sigma_{22} \vartheta_2^2 }.
$$
Taking the inverse Laplace transform, the solution of the problem has the following form: 
$$f(x,y)= \rho'^{\frac{\pi}{\theta}} \sin \left(\frac{\vartheta' \pi}{\theta}\right),$$
where $\phi(x,y)=(x',y')$ (the transformation $\phi$ being defined in \eqref{eq:change_variable}), $x'=\rho' \cos (\vartheta')$ and $y= \rho' \sin(\vartheta')$. The function $f(x,y)$ defined above is indeed the unique continuous harmonic function (see \eqref{eq:expression_continuous_harmonic_function}), up to multiplicative constants.

\section*{Acknowledgments}
\setcounter{equation}{0}

We wish to thank D.~Denisov, G.~Fayolle, R.~Garbit, I.~Ignatiuk-Robert, I.~Kurkova, M.~Peign\'e, L.~V\'eron, V.~Wachtel, W.~Woess for useful discussions. We also thank an associate editor and a referee for their comments and suggestions, which led us to improve the presentation of the paper.

\end{document}